\documentclass[11pt]{amsart}
\usepackage{amssymb}
\usepackage{amsmath}
\usepackage{amscd}
\usepackage{curves}
\usepackage{epsfig}
\usepackage{bbm}
\usepackage{geometry}
\usepackage{graphicx}
\usepackage{pstricks}
\usepackage{pstricks-add}
\usepackage{pst-grad}
\usepackage{pst-plot}
\usepackage{verbatim}
\usepackage{latexsym}
\usepackage{eucal}
\usepackage{amsfonts}
\usepackage{enumerate}
\numberwithin{equation}{section}
\newtheorem{theorem}{Theorem}[section]
\newtheorem{lemma}[theorem]{Lemma}
\newtheorem{prop}[theorem]{Proposition}
\newtheorem{corollary}[theorem]{Corollary}

\def \bpf {\begin{proof}}
\def \epf {\end{proof}}
\def \beq {\begin{equation*}}
\def \eeq {\end{equation*}}
\def \bsp{\begin{split}}
\def \esp{\end{split}}

\def \hone  {{\dot{H}^{1}}}

\def \CI {{C^\infty}}

\def \vtheta {\vartheta}

\def \mca {{\mathcal A}}
\def \mcb {{\mathcal B}}
\def \mcc {{\mathcal C}}

\def \mcf {{\mathcal F}}
\def \mcg {{\mathcal G}}

\def \mci {{\mathcal I}}

\def \mcl {{\mathcal L}}
\def \mcm {{\mathcal M}}
\def \mcw {{\mathcal W}}
\def \mco {{\mathcal O}}
\def \mcp {{\mathcal P}}
\def \mcq {{\mathcal Q}}
\def \mcr {{\mathcal R}}

\def \mcz {{\mathcal Z}}

\def \mbr {{\mathbb R}}
\def \mr {{\mathbb R}}
\def \mn {{\mathbb N}}
\def \ms {{\mathbb S}}
\def \mbs {{\mathbb S}}

\def\ha {\frac{1}{2}}
\def \tha  {\frac32}

\def \oq {\frac{1}{4}}

\def \ds {\displaystyle}

\def \loc {\operatorname{loc}}

\def \mcn {\mathcal{N}}

\def \ga {{\gamma}}
\def \eps {\varepsilon}   
\def \ups {\upsilon}   
\def \Ups {\Upsilon}   
   
\def \vphi {\varphi}

\def \vtheta {\vartheta}

\def \zed {{\mathcal{Z}}}

\def \veps {{\vec{\eps}}}
\def \vepa {{\vec{\eps}\,{}^\alpha}}

\def \La {\Lambda}   
\def \lan {\langle}   
\def \ran {\rangle}   
\def \del {\delta}   
\def \lap {\Delta}
\def \p {\partial}

\def \beqq {\begin{equation}}
\def \eeqq {\end{equation}}

\numberwithin{equation}{section}

\begin{document}
\title[Inverse Scattering for Semilinear Wave Equations]{Inverse Scattering for Critical Semilinear Wave Equations}
\author{Ant\^onio S\'a Barreto}
\address{Ant\^onio S\'a Barreto\newline
\indent Department of Mathematics, Purdue University \newline
\indent 150 North University Street, West Lafayette, Indiana,  47907, USA}
\email{sabarre@purdue.edu}
\author{Gunther Uhlmann}
\address{Gunther Uhlmann \newline\indent  Department of Mathematics, University of Washington, Seattle, WA 98195, \newline\indent 
IAS, HKUST, Clear Water Bay, Hong Kong, China}
\email{gunther@math.washington.edu}
\author{Yiran Wang}
\address{Yiran Wang\newline \indent Department of Mathematics Emory University \newline \indent 400 Dowman Drive
Atlanta, GA 30322}
\email{yiran.wang@emory.edu}
\dedicatory{\today}
\begin{abstract}   We show that the scattering operator for defocusing energy critical 
semilinear  wave equations $\square u+f(u)=0,$ $f\in C^\infty(\mr)$ and $f\sim u^5,$  in three space dimensions, determines $f$.
\end{abstract}
\maketitle
\tableofcontents
%%%%%%%%%%%%%%
\section{Introduction}
We consider the following question:  Given a nonlinear wave equation for which there is a well defined scattering operator, even if for only small data, what information about the equation can be extracted from its scattering operator?   The similar question for linear equations has been very well studied, even though there are many important unanswered questions,  see for example the surveys \cite{Uhl,Uhl1} and references cited there for a general discussion about linear inverse scattering problems, but there seems to be relatively few results for nonlinear inverse scattering problems, see for example \cite{CarGal,Fur,MorStr,PauStr} and references cited there.

We study inverse scattering for semilinear wave equations of the form
\begin{equation}
\begin{split}
&\left(\p_t^2 -  \lap\right)u + f(u)= 0, \ \ (t, x) \in (0, \infty)\times \mbr^3, \\
 & \;\ \;\ \;\ \;\ \;\ u(0, x) = \vphi(x), \ \ \p_t u(0, x) = \psi(x),
\end{split} \label{Weq}
\end{equation}
where $\Delta=\sum_{j=1}^3 \p_{x_j}^2$ is the negative Euclidean Laplacian.  The nonlinear term is of the form $f(u)=u h(u),$ so one can think of $h(u)$ as a nonlinear potential.    The problem is, given $f(u)$ such that the scattering operator  is well defined, can one determine $f(u)$ from the scattering operator?

 We shall focus our attention to the case where \eqref{Weq} has global solution for arbitrary  finite energy data and a well defined global scattering operator.  This occurs when $f(u) \sim u^5$ as $|u|\nearrow \infty$ or $|u|\searrow 0$ -- see the precise assumptions {\bf H.\ref{H1}} to {\bf H.\ref{H4}} below.    Under these conditions, the existence and uniqueness of global solutions to \eqref{Weq} was established in a series of papers by several authors, beginning with the work of Strauss \cite{Str}, Rauch \cite{Rau}, Struwe  \cite{Str} and Grillakis \cite{Gri1}, and finally  Shatah and Struwe \cite{ShaStr1,ShaStr2} established the global existence and uniqueness of weak solutions in the space $X_{\loc}(\mr,\mr^3)$ defined in \eqref{defX} below.  These became known as the Shatah-Struwe solutions. Masmoudi and Planchon  \cite{MasPla} have relaxed the necessary condition for uniqueness, but uniqueness is not known for other possible weak solutions with finite energy. Bahouri and G\'erard \cite{BahGer} proved asymptotic completeness for the  Shatah-Struwe solutions and defined the semilinear scattering operator.

The problem of determining a linear potential from the scattering operator has a very long history beginning with the work of  Faddeev \cite{Fad}. In the case of time dependent potentials this problem was considered by Stefanov \cite{Ste}.    Morawetz and Strauss \cite{MorStr} studied the inverse scattering problem for the Klein-Gordon equation and proved that, under certain conditions, the scattering operator determines a real analytic potential $f(u)=u h(u).$  Bachelot \cite{Bac} extended their results to the case where the potential is of the form $f(x,u),$ which is real analytic in $u.$  Carles and Gallagher \cite{CarGal} proved results similar to the ones by Morawetz and Stauss for the several dispersive equations, including the nonlinear Schr\"odinger, wave  and Klein Gordon equation with critical nonlinearities.     Furuya \cite{Fur} has recently studied the related problem, also for the real analytic nonlinearities, but on the frequency side, for the Helmholtz equation.  Pausader and Strauss \cite{PauStr} studied inverse scattering for the fourth-order nonlinear wave equation, or the Bretherton equation. Sasaki studied inverse scattering for Hartree equation \cite{Sas1} and for the Schr\"odinger equation with the Yukawa potential \cite{Sas2}.  The common point of these papers is that the scattering operator in these settings is analytic, see for example Definition 1.1 of \cite{CarGal}, for an explanation.

In the case discussed here the potential is not real analytic.  The novelty is that we use microlocal analysis methods to study the propagation of singularities generated by the interaction of semilinear conormal waves to show that, under hypothesis  {\bf H.\ref{H1}} to {\bf H.\ref{H4}}  and {\bf B.\ref{H4N}} below,  the scattering operator uniquely determines $f(u)$ fo all $u\in \mr.$   Kurylev, Lassas and Uhlmann \cite{KLU} were the first to use singularities generated by the interaction of nonlinear conormal waves to study inverse problems for nonlinear wave equations.

\section{Preliminaries}

 The scattering operator for the semilinear equation \eqref{Weq} is defined by comparing the asymptotic behavior of  the  Shatah-Struwe of solutions of \eqref{Weq}  to the asymptotic behavior of solutions of the Cauchy problem for the linear wave equation
\begin{equation}
\begin{split}
&\left(\p_t^2-\Delta\right)v(t,x)=0,  \\
& v(0,x)=\vphi(x), \;\ \p_t v(0,x)= \psi(x).
\end{split}\label{lweq}
\end{equation}

Both equations \eqref{Weq} and \eqref{lweq} have conserved energies.  The energy form of the solution of the linear wave equation \eqref{lweq} is given by
\begin{equation}\label{lenergy}
E_0(v, \p_t v)(t)= \ha \int_{\mr^3} \left( |\p_t v(t,x)|^2 + |\nabla_x v(t,x)|^2 \right) \; dx,
\end{equation} 
while the energy form of the solution of the semilinear wave equation \eqref{Weq}, is  given by
\begin{gather}
\begin{gathered}
E(u, \p_t u)(t)= \ha \int_{\mr^3} \left( |\p_t u(t,x)|^2 + |\nabla_x u(t,x)|^2 + F(u(t,x)\right) \; dx,\\
\text{ where } F(u)= \int_0^u f(s) ds,
\end{gathered}\label{energy}
\end{gather}
and in both cases we have
\begin{gather}
\begin{gathered}
E_0(v, \p_t v)(t)= E_0(v, \p_t v)(0), \\
E(u, \p_t u)(t)= E(u, \p_t u)(0). 
\end{gathered}\label{cons}
\end{gather}

One has to make some assumptions on $f(u)$ so that one can define a scattering operator for \eqref{Weq}. As in \cite{BahGer,BasSaB,ShaStr1},  we shall assume that $f(u)$ satisfies the following hypothesis: \\
 \begin{enumerate}[\bf H1.]
 \item \label{H1} $f(u)=u h(u),$ $h$ is even (or $h(u)= h_0(u^2)$) and  $\frac{1}{C}|u|^4\leq |h(u)| \leq C|u|^4$ for all $u\in \mr.$ \\
\item\label{H2} $u f'(u) \sim f(u)$ \text{ as } $|u|\nearrow \infty$   and as  $|u| \searrow 0.$\\
\item \label{H3}  The function $\ds F(u)=\int_0^u f(s) ds$  is convex.\\
\item  \label{H4}  There exit $C_j>0$ such that $|f^{(j)}(u)|\leq C_j|u|^{5-j},$  $0\leq j\leq 5.$ \\
\end{enumerate}
To be able to solve the inverse problem, we will add the following assumption:\\
\begin{enumerate}[\bf B1.]
\item\label{H4N}   $f(u)$  is such that  $f^{(4)}(u)= 0$  if and only if  $u=0.$\\
\end{enumerate}

 For example, $f(u)=u^5,$  satisfies {\bf B.\ref{H4N}}, and in general perturbation of $f(u)=u^5$ such that $f^{(4)}(u)= Cu(1+ Z(u)),$  $C>0,$ $Z(u)$ even, $|Z(u)|<1,$ and   $f^{(j)}(0)=0,$ $j=0,1,2,3,$  then $f(u)$ satisfies 
 {\bf H.\ref{H1}} to {\bf H.\ref{H4}} and {\bf B.\ref{H4N}}.

 Under assumptions {\bf H.\ref{H1}} to {\bf H.\ref{H4}} on the function $f,$  Shatah and Struwe \cite{ShaStr1,ShaStr2} and  Bahouri and G\'erard \cite{BahGer} examined the well posedness of \eqref{Weq} for solutions in the spaces 
\begin{gather}
\begin{gathered}
X(\mr;\mr^3)= C^0(\mr; {\dot H}^1(\mr^3) )\cap C^1(\mr; L^2(\mr^3))\cap L^5(\mr; L^{10}(\mr^3)), \text{ and } \\
X_{\loc}(\mr; \mr^3)= C^0(\mr; {\dot H}^1(\mr^3) )\cap C^1(\mr; L^2(\mr^3))\cap L_{\loc}^5(\mr; L^{10}(\mr^3)), \\
\text{ where } \\
{\dot{H}}^{1}(\mr^3)= \{ v:  \; ||v||_{\dot H^1}^2= \int_{\mr^3} |\nabla_z v|^2 \, dx<\infty\}, \text{ and } \\
   L^p(\mr_t; L^{q}(\mr^3))= \{ u(t,x): ||u||_{L^p;L^{q}}^q= \int_{\mr}\left(\int_{\mr^3} |u(t,x)|^{q} \; dx\right)^{\frac{p}{q}} \; dt<\infty \;\ \},
\end{gathered} \label{defX}
   \end{gather} 
and proved the following:

\begin{theorem}[Shatah and Struwe \cite{ShaStr1}, Bahouri and G\'erard \cite{BahGer}]\label{ShaStru}
 Under the conditions {\bf H.\ref{H1}} to {\bf H.\ref{H4}}, for any $(\vphi, \psi) \in \hone(\mr^{3}) \times L^{2}(\mr^{3})$, there exists a unique $u \in X(\mr; \mr^3)$ satisfying
\eqref{Weq}.
\end{theorem}

Shatah and Struwe \cite{ShaStr1,ShaStr2} showed the existence and uniqueness for $u \in X_{\loc}(\mr; \mr^3)$ and  Bahouri and G{\'e}rard \cite{BahGer} showed that in fact such solutions $u \in X(\mr; \mr^3).$  These are known as the Shatah-Struwe solutions of \eqref{Weq}.

One of the key points in the proof of Theorem \ref{ShaStru} is the following Strichartz estimate:
 \begin{theorem}[Ginibre and Velo \cite{GinVel}]\label{GV}  Given $r\in [6,\infty)$, let $q$ satisfy
 \begin{gather*}
 \frac{1}{q} + \frac{3}{r} = \frac{1}{2}.
 \end{gather*}
 Then there exists $C_{r}$ such that for every $w(t,z)$ defined on $\mr \times \mr^{3}$, and for any $T,$
\begin{gather}
 ||w||_{L^q([0,T)); L^r(\mr^3))} \leq  C_{r}\left( E_0(w,\p_t w)(0) +  ||\square w||_{L^1([0,T)); L^2(\mr^3))} \right). \label{STES0}
    \end{gather}
  \end{theorem}
But the usual energy method easily shows that
\begin{gather*}
E_0( w,\p_t w)^2(t)=\int_{\mr^3}\left( |\nabla_x w(t,x)|^2 + |\p_t w(t,x)|^2\right) dx
\end{gather*}
satisfies
\begin{gather}
E_0(w,\p_t w)(T)\leq E_0(w,\p_t w)(0) + || \square w||_{L^1([0,T]; L^2(\mr^3))},\label{STES0-1}
\end{gather}
and therefore, for $t\in [0,T],$  we have
\begin{gather}
 ||w||_{L^q([0,T]); L^r(\mr^3))} + E(w,\p_t w)(T) \leq  C_r\left( E_0(w,\p_t w)(0) +  ||\square w||_{L^1([0,T]); L^2(\mr^3))} \right), \label{STES}
 \end{gather}
we kept the notation $C_r,$ for a constant which depends only on $r.$

Bahouri and G\'erard \cite{BahGer}   proved the asymptotic completeness for solutions to \eqref{Weq} and defined the corresponding scattering operator.  %We recall the proof of their theorem in Section \ref{BAH}. 
They showed that given finite energy Cauchy data $(\vphi,\psi),$ and if $u$ is the corresponding Shatah-Struwe solution to \eqref{Weq}, there exist $(\vphi_0^\pm,\psi_0^\pm)\in L^2(\mr^3) \times {\dot{H}}^1(\mr^3)$   such that  if $v^\pm(t,x)$ are the solutions of the  Cauchy problem for the linear wave equation \eqref{lweq}  with initial data 
$(\vphi_0^\pm,\psi_0^\pm),$ then
\begin{equation}
\begin{split}
& \lim_{t\rightarrow \infty} E_0( v^+(t)- u(t), \p_t( v^+(t)- u(t))) =0, \\
& \lim_{t\rightarrow -\infty} E_0( v^-(t)- u(t), \p_t( v^-(t)- u(t))) =0.
\end{split}
\end{equation}
They also proved that  the maps
\begin{equation}
\begin{split}
\Omega_\pm : \;  {\dot{H}}^1(\mr^3) & \times L^2(\mr^3)  \longrightarrow {\dot{H}}^1(\mr^3) \times L^2(\mr^3) \\
& (\vphi_0^\pm,\psi_0^\pm) \longmapsto (\vphi,\psi)
\end{split}\label{acomp}
\end{equation}
are isometries.

  As usual, the nonlinear scattering operator was then defined by Bahouri and G\'erard \cite{BahGer} as the map
\begin{equation}
\begin{split}
\mcm: \;  {\dot{H}}^1(\mr^3) \times & L^2(\mr^3)   \longrightarrow {\dot{H}}^1(\mr^3) \times L^2(\mr^3) \\
& \mcm= \Omega_+ \circ \Omega_-^{-1}.
\end{split}\label{Moe}
\end{equation}
The operators $\Omega_\pm$ are known as the nonlinear M{\o}ller wave operators and $\mcm$ as the nonlinear M{\o}ller scattering operator.

  We will prove the following:
\begin{theorem}\label{main} Let $f_j(u)$ $j=1,2$ satisfy hypotheses {\bf H.\ref{H1}} to {\bf H.\ref{H4}}  and {\bf B.\ref{H4N}}. 
 Let $\mcm_j$ be the nonlinear M{\o}ller scattering operator defined in \eqref{Moe} associated with $f_j.$  If $\mcm_1=\mcm_2,$ then $f_1(u)=f_2(u)$ for all $u\in \mr.$
\end{theorem}

The proof of Theorem \ref{main} relies on   the analysis of the  propagation of conormal singularities for solutions of semiliear wave equation. However, the M{\o}ller scattering operator defined in \eqref{Moe} is not quite suitable for the study of propagation of singularities, and as in \cite{BasSaB}  we will rephrase it in terms of Friedlander radiation fields.

\subsection{The Radiation Fields and the Scattering Operator}

  In the case  of the linear wave equation \eqref{lweq}, the forward and backward radiation fields  for the wave equation with a forcing term $f(t,x)\in C_0^\infty(\mr \times \mr^3),$ 

\begin{equation}
\begin{split}
\left(\p_t^2-\Delta\right) v(t,x) & =f (t,x),\\
 v(0,x)=\vphi(x), \;\ &  \p_t v(0,x)= \psi(x), \;\ \vphi,\psi\in C_0^\infty(\mr^3), \label{Weq-F}
\end{split}
\end{equation}
are defined to be  respectively
\begin{gather}
\begin{gathered}
\mcr_+(\vphi, \psi,f)(s, \theta) = \lim_{r\rightarrow \infty} r(\p_t v)(s+r, r\theta), \\
\mcr-(\vphi, \psi,f)(s, \theta) = \lim_{r\rightarrow \infty} r(\p_t v)(s-r, r\theta),
\end{gathered}\label{radf}
\end{gather}
where $r=|x|$ and $\theta=\frac{x}{|x|}.$

In general, when we are not referring to initial data or forcing term, we denote the forward and backward radiation fields of a  function $u(t,r,\omega),$  when they exist, by 
\begin{gather}
\begin{gathered}
\mcn_+ u(s,\omega)= \lim_{r\rightarrow \infty} r \p_s u(s+r,r,\omega) \text{ and } 
\mcn_- u(s,\omega)= \lim_{r\rightarrow \infty} r \p_s u(s-r,r,\omega).
\end{gathered}\label{radf-func}
\end{gather}

It is well-known, see for example \cite{BasSaB} for a proof, the limits \eqref{radf} can be computed in terms of the Radon transform of the initial data and the forcing term:
\beqq\label{linrad}
\begin{gathered}
\mcr_+(\vphi, \psi,f)(s, \theta) = -\frac{1}{4\pi} \p_s\bigg(R\psi(s, -\theta) + \p_s R\vphi(s, -\theta)+  \int_{t - \langle\theta, z\rangle = s}H(t) f(t,x) d\sigma(t, z)\ \bigg)\\
\mcr_-(\vphi, \psi,f)(s, \theta) = \frac{1}{4\pi} \p_s\bigg(R\psi(s,  \theta) + \p_s R\vphi(s,  \theta)- \int_{t + \langle\theta, z\rangle = s} H(-t) f(t,x) d\sigma(t, z) \bigg),
\end{gathered}
\eeqq
where $H(t)$ is the Heaviside function,  and $\sigma(t,x)$ is the corresponding surface measure and $R$ is the Radon transform:
\begin{equation*}
R g(s,\theta)= \int_{\lan x,\theta\ran=s} g(x) d \mu(x) \\
\text{ and } \mu(x) \text{ is the surface measure on the plane } \lan z,\theta\ran=s,
\end{equation*}
see for example \cite{LaxPhi} and \cite{Fried1}.  These maps have an extension as bounded operators 
\begin{equation}
\begin{split}
\mcr_\pm:  \;  \dot H^1(\mbr^3)\times  L^2(\mbr^3)& \times  L^1(\mr; L^2(\mr^3)) \rightarrow L^2(\mbr\times \mbs^2), \\
& (\vphi, \psi, f) \longmapsto \mcr_\pm(\vphi, \psi, f), 
\end{split}\label{bddnes}
\end{equation}
with the Lebesgue measure on $\mr\times \ms^2,$ and moreover,  the maps
\begin{equation}
\begin{split}
\mcr_\pm: \; & \dot H^1(\mbr^3)\times L^2(\mbr^3) \rightarrow L^2(\mbr\times \mbs^2), \\
& (\vphi, \psi) \longmapsto \mcr_\pm(\vphi, \psi,0), 
\end{split}\label{defradf}
\end{equation}
are unitary, in the sense that
\begin{gather}
E_0(\vphi,\psi)= ||\mcr_\pm (\vphi,\psi,0)||_{L^2(\mr \times \ms^2)}^2. \label{liniso}
\end{gather}

One also has the following inequalities which will be very useful below:
\begin{gather}
\begin{gathered}
||\mcr_+(\vphi,\psi,f)||_{L^2(\mr\times \ms^2)} \leq E_0(\vphi,\psi)+ ||f||_{L^1([0,\infty);L^2(\mr^3)}, \\
||\mcr_-(\vphi,\psi,f)||_{L^2(\mr\times \ms^2)} \leq E_0(\vphi,\psi)+ ||f||_{L^1([-\infty,0);L^2(\mr^3)},
\end{gathered} \label{est-radf}
\end{gather}
 see the proof of Theorem 2.1 of \cite{BasSaB}.

The Friedlander radiation fields can also be defined for the semilinear wave equation \eqref{Weq} with $f(u)$ satisfying 
{\bf H\ref{H1}}-{\bf H.\ref{H4}}.  This was shown by Grillakis \cite{Gri1} for initial data $\vphi,\psi \in C_0^\infty.$     Baskin and S\'a Barreto \cite{BasSaB}  showed that the maps
\begin{gather}
\begin{gathered}
\mcl_+(\vphi, \psi)(s, \theta) = \lim_{r\rightarrow \infty} r(\p_t u)(s+r, r\theta), \\
\mcl_-(\vphi, \psi)(s, \theta) = \lim_{r\rightarrow \infty} r(\p_t u)(s-r, r\theta),
\end{gathered}\label{rsladf}
\end{gather}
where $u(t,x)$ is the Shatah-Struwe solution of \eqref{Weq} with initial data $(\vphi,\psi)\in C_0^\infty(\mr^3) \times C_0^\infty(\mr^3),$ extend to  nonlinear isomorphisms
\begin{gather}
\begin{gathered}
\mcl_\pm: \dot H^1(\mbr^3)\times L^2(\mbr^3) \rightarrow L^2(\mbr\times \mbs^2),\\
\mcl_\pm(\phi,\psi)=\mcr_\pm( \phi,\psi,-f(u)),
\end{gathered}\label{nonl-Iso}
\end{gather}
but now with the semilinear energy norm
\begin{gather}
E(\vphi,\psi)= ||\mcl_\pm (\vphi,\psi)||_{L^2(\mr \times \ms^2)}^2. \label{iso}
\end{gather}

This relies on the key ingredient used to establish the existence and uniqueness of global solutions, which was proved by Bahouri and Shatah \cite{BahSha}:
\begin{gather*}
\lim_{t\rightarrow \infty} \int F(u(t,x))\ dx=0.
\end{gather*}

Therefore the Friedlander nonlinear scattering operator
\begin{gather}
\mca = \mcl_+\circ \mcl_-^{-1}: L^2(\mr \times \ms^2) \longmapsto L^2(\mr\times \ms^2) \label{frscat}
\end{gather}
is an isometry.  It is shown in \cite{BasSaB} that the Friedlander and the  M{\o}ller scattering operators are related by 
\begin{gather*}
\mca= \mcr_+ \circ \mcm \circ \mcr_-^{-1},
\end{gather*}
where $\mcr_\pm$ are defined in \eqref{defradf}.

Therefore Theorem \ref{main} is equivalent to
\begin{theorem}\label{main1} Let $f_j(u)$ $j=1,2$ satisfy hypotheses {\bf H.\ref{H1}} to  {\bf H.\ref{H4}} and  {\bf B.\ref{H4N}}.  Let $\mca_j$ be the Friedlander scattering operator defined in \eqref{frscat} associated with $f_j.$  If $\mca_1=\mca_2,$ then $f_1(u)=f_2(u)$ for all $u\in \mr.$
\end{theorem}

\subsection{A geometric interpretation of the radiation fields}  It is useful to give a geometric interpretation of the limits \eqref{rsladf} as the asymptotic behavior of the solution $u$ to either \eqref{Weq} or 
 \eqref{Weq-F}  on a compact manifold with corners.
  This is the point of view used by Baskin, Vasy and Wunsch \cite{BasVasWun1,BasVasWun} to analyze the asymptotic behavior of the radiations fields as $|s|\nearrow \infty$ on non-trapping asymptotically Minkowski manifolds.  First one uses stereographic projection to compactify $\mr\times \mr^3$  into the upper hemisphere of $\ms^4$ by setting
\begin{gather*}
\mcb: \mr \times \mr^3 \longrightarrow \ms^4_+\\
(t,x) \longmapsto \frac{1}{(1+|x|^2+t^2)^\ha}(1,t,x).
 \end{gather*}
 The sphere $\ms^3,$ which is the boundary of $\ms_+^4,$ corresponds to $\rho=(1+|x|^2+t^2)^{-\ha}=0.$   Perhaps it is more intuitive to think of
 \begin{gather*}
 (T,X)= \frac{1}{(1+|x|^2+t^2)^\ha}(t,x) \in \{z\in \mr^4: |z|<1\},
 \end{gather*}
 and think of $\ms_+^4$ as the interior of the unit ball in $\mr^4.$
 Notice that a light cone with vertex $(t_0,x_0),$ which is given by $ |t-t_0|^2=|x-x_0|^2$ will  always intersect the boundary at infinity at the manifolds
 \begin{gather*}
S_+=\{\rho=0, \;\   T=|X|\} \text{ and } S_-=\{\rho=0, \;\ T=-|X|\},
 \end{gather*}
independently of its vertex, see Fig.\ref{Fig0}.

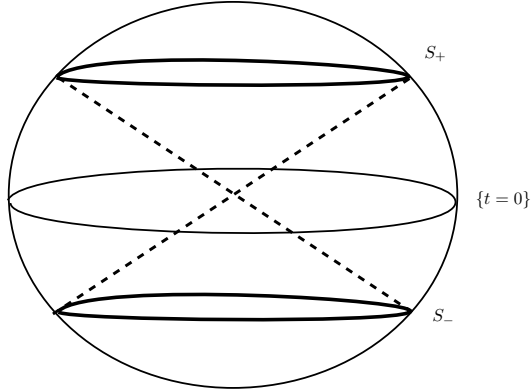
\begin{figure}[h!]
% Generated with LaTeXDraw 2.0.8
% Fri Feb 07 14:25:29 EST 2020
% \usepackage[usenames,dvipsnames]{pstricks}
% \usepackage{epsfig}
% \usepackage{pst-grad} % For gradients
% \usepackage{pst-plot} % For axes
\scalebox{.6} % Change this value to rescale the drawing.
{
\begin{pspicture}(0,-4.3)(11.721894,4.3)
\psellipse[linewidth=0.04,dimen=outer](4.99,0.0)(4.99,4.3)
\psbezier[linewidth=0.04](0.06,-0.18)(0.06,-0.98)(9.918998,-1.1599995)(9.92,-0.16)(9.921002,0.8399995)(0.3845453,0.8031061)(0.0,-0.12)
\psbezier[linewidth=0.08](1.08,2.6)(1.4,2.42)(8.42,2.34)(8.86,2.6)(9.3,2.86)(1.7898004,3.3244028)(1.08,2.62)
\psbezier[linewidth=0.08](1.12,-2.6)(1.44,-2.78)(8.46,-2.86)(8.9,-2.6)(9.34,-2.34)(1.8298005,-1.8755972)(1.12,-2.58)
\usefont{T1}{ptm}{m}{n}
\rput(9.471455,3.125){$S_+$}
\usefont{T1}{ptm}{m}{n}
\rput(9.651455,-2.715){$S_-$}
\psline[linewidth=0.066cm,linestyle=dashed,dash=0.16cm 0.16cm](8.9,2.58)(1.0,-2.64)
\psline[linewidth=0.066cm,linestyle=dashed,dash=0.16cm 0.16cm](1.12,2.6)(8.96,-2.6)
\usefont{T1}{ptm}{m}{n}
\rput(10.981455,-0.115){$\{t=0\}$}
\end{pspicture} 
}
\caption{A compactification of Minkowski space ${\mr\times \mr^3}.$ All light cones, independently of where their vertices are located, intersect the boundary at infinity along $S_\pm.$}
\label{Fig0}
\end{figure}

The next step is to blow-up the ball $B=\{z\in \mr^4: |z|\leq 1\}$ along the manifolds $S_\pm.$  We then consider the pull back the function $u,$ which solves \eqref{Weq} or \eqref{Weq-F}, to the manifold obtained by blowing  up $B$ along $S_\pm$ and analyze the limit of the pull-back of $u$ at the boundary faces introduced by this blow-up.  We denote these faces by $\mcf_+,$  and they are usually called the front face, see Fig.\ref{Fig1}.   In this region one can  can use projective coordinates
\begin{gather*}
s=\frac{T-X}{\rho},  T, \text{ and } R=\rho \text{ near } S_+, \text{ in the region where } |s|<\infty, \\
s=\frac{T+X}{\rho}, T, \text{ and } R=\rho \text{ near } S_-, \text{ in the region where } |s|<\infty.
\end{gather*}
Notice that the new boundary faces introduced by this blow-up is $R=0$ in both sets of coordinates.  But if one writes this in terms of the original variables $(t,x)$
one finds that
\begin{gather*}
s=t-x, \;\ R=(1+t^2+|x|^2)^{-\ha},  \text{ near } S_+, \text{ in the region where } |s|<\infty, \\
s=t+x, \;\ R=(1+t^2+|x|^2)^{-\ha},  \text{ near } S_-, \text{ in the region where } |s|<\infty.
\end{gather*}
But in the region where $|s|$ is finite, $R\rightarrow 0$ if $r=|x|\rightarrow \infty.$ Therefore, the limits \eqref{radf-func} can be thought of as the limit of the lift of the function $u$ to the blown-up manifold at the new boundary face introduced by the blow-up of $S_\pm.$ 

 Given an $L^2$ function on one of the new boundary faces $\mcf_+$ or $\mcf_-,$  there exists a unique  solution of \eqref{Weq} with such forward (if data is placed in $\mcf_+$) or backward (if data is placed on $\mcf_-$) radiation field.  The map that takes the data on $\mcf_-$ to the limit on $\mcf_+$ is the scattering operator. The inverse of the scattering operator, of course does the opposite.

\begin{figure}[h!]
% Generated with LaTeXDraw 2.0.8
% Fri Feb 07 13:59:16 EST 2020
% \usepackage[usenames,dvipsnames]{pstricks}
% \usepackage{epsfig}
% \usepackage{pst-grad} % For gradients
% \usepackage{pst-plot} % For axes
\scalebox{.6} % Change this value to rescale the drawing.
{
\begin{pspicture}(0,-4.128254)(11.561894,4.1282544)
\psbezier[linewidth=0.04](1.62,2.7465923)(1.8,3.686592)(3.7402346,4.10493)(4.72,4.106592)(5.699765,4.1082544)(7.86,3.9665923)(8.22,2.7865922)
\psline[linewidth=0.04cm](1.6,2.7665923)(0.04,1.1465921)
\psline[linewidth=0.04cm](8.24,2.7465923)(9.84,1.1665922)
\psline[linewidth=0.04cm](0.04,1.1465921)(0.0,-1.2134078)
\psline[linewidth=0.04cm](9.86,1.1465921)(9.86,-1.2534078)
\psline[linewidth=0.04cm](0.02,-1.2534078)(1.62,-2.8334079)
\psline[linewidth=0.04cm](9.84,-1.3134078)(8.28,-2.9334078)
\psbezier[linewidth=0.04](8.28,-2.9534078)(8.111803,-3.7714646)(6.0041723,-4.0785613)(5.0,-4.0934076)(3.9958274,-4.1082544)(1.9874172,-3.8766704)(1.6297125,-2.8534079)
\psbezier[linewidth=0.04](1.66,2.7465923)(1.38,2.446592)(7.78,2.2865922)(8.16,2.7865922)(8.54,3.2865922)(2.2910435,3.6256325)(1.64,2.8665922)
\psbezier[linewidth=0.04](1.6198823,-2.9293077)(2.08,-3.6534078)(8.240353,-3.4734077)(8.300177,-2.9973202)(8.36,-2.5212326)(1.9714543,-2.1134079)(1.58,-2.8953013)
\psbezier[linewidth=0.04](0.0597577,-1.223214)(0.04,-1.4134078)(9.88,-1.5334078)(9.86003,-1.2434142)(9.840061,-0.9534206)(0.66,-0.9134078)(0.02,-1.1626136)
\psbezier[linewidth=0.04](0.1,1.0265921)(0.0,0.9265922)(9.259,0.78659266)(9.840009,1.1265922)(10.421017,1.4665917)(0.34,1.4065922)(0.08,1.1465921)
\psbezier[linewidth=0.062](0.11960783,-0.0818509)(0.02,-0.4029961)(9.600956,-0.6534078)(9.841332,-0.0461681)(10.081709,0.56107163)(0.50711644,0.7865922)(0.059843134,-0.010485301)
\usefont{T1}{ptm}{m}{n}
\rput(9.541455,2.2115922){$\mcf_+$}
\usefont{T1}{ptm}{m}{n}
\rput(9.761456,-2.4884079){$\mcf_-$}
\usefont{T1}{ptm}{m}{n}
\rput(10.821455,-0.028407807){$\{t=0\}$}
\psline[linewidth=0.062cm,linestyle=dashed,dash=0.16cm 0.16cm](1.04,2.206592)(8.88,-2.2334077)
\psline[linewidth=0.062cm,linestyle=dashed,dash=0.16cm 0.16cm](0.98,-2.2734077)(8.7,2.2665923)
\psbezier[linewidth=0.062,linestyle=dashed,dash=0.16cm 0.16cm](1.12,2.1665921)(1.08,2.1265922)(8.16,1.8665922)(8.66,2.2865922)(9.16,2.706592)(1.6034849,3.098627)(1.08,2.2465923)
\psbezier[linewidth=0.062,linestyle=dashed,dash=0.16cm 0.16cm](1.02,-2.2334077)(0.68,-2.5334077)(8.4,-2.7334077)(8.94,-2.2134078)(9.48,-1.6934078)(1.6378766,-1.3632692)(1.0,-2.1334078)
\end{pspicture} 
}
\caption{The manifold with corners obtained by blowing up the boundary of the compactification of ${\mr\times \mr^3}$ along $S_\pm,$ and the asymptotic behavior of a cone in Minkowski space. Following Friedlander \cite{Fried1,Fried2}, this can be thought of as hourglass waves  starting at negative infinity, collapsing to a point, reemerging for $t>0$ and expanding back to infinity.}
\label{Fig1}
\end{figure}
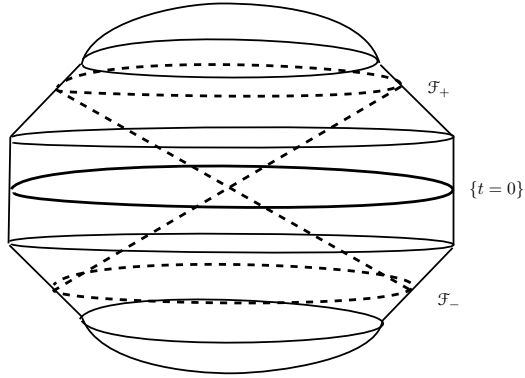

\subsection{Considerations about the proof of Theorem \ref{main1}}The  proof of Theorem \ref{main1} has two main ingredients: 
\begin{enumerate}[1.]
\item The linearization of the nonlinear scattering operator $\mca$ at a non-zero $C^\infty$ solution with a given radiation field, as in the work of Carles and Gallagher \cite{CarGal}.  The Strichartz estimates of Ginibre and Velo \eqref{STES} are used to control the error terms of the linear approximation
\item The analysis of the singularities formed by the interaction of three and four transversal progressing waves.    We show that the principal symbol of the radiation field  determine the nonlinearity on the manifold where the waves interact and this is enough to determine $f(u).$
\end{enumerate}

The idea of using nonlinear singularities to study inverse problems was first used by Kurylev, Lassas and Uhlmann \cite{KLU} and further developed by Lassas, Uhlmann and Wang \cite{LUW}, Uhlmann and Wang \cite{UhlWan} and  X.  Chen, M. Lassas, L. Oksanen and G. Paternain \cite{Paternain},  Feizmohammadi and Oksanen \cite{FeiOks}.

 The study of propagation of singularities  for semilinear wave equations started with the work of Bony \cite{Bon}.  The more specific question of propagation of conormal singularities for semilinear wave equations was studied by several authors, including  Beals \cite{Beals4}, Bony \cite{Bony3,Bony6,Bony1},  Chemin \cite{Che}, Melrose and Ritter \cite{MelRit},  Rauch and Reed \cite{RauRee}, S\'a Barreto \cite{SaB} and  more recently by S\'a Barreto and Wang \cite{SaWang1} and S\'a Barreto \cite{SaB2}.

We end this section by recalling the following support and regularity theorem for semilinear radiation fields, Theorem 1.2 of \cite{BasSaB}:
\begin{theorem}(Baskin and S\'a Barreto \cite{BasSaB}) \label{sup-thm} Let $F\in L^2(\mr)$ be compactly supported satisfy $\ds \int_\mr F(s) ds=0.$ If $(\vphi,\psi)\in \dot H^1(\mbr^3)\times  L^2(\mbr^3)$ are such that $\mcl_\pm(\vphi,\psi)=F,$ then $\vphi$ and $\psi$ are compactly supported and radial.  Moreover, if $F\in C_0^\infty(\mr),$  then $\vphi,\psi\in C_0^\infty(\mr^3).$ 
\end{theorem}

\section{The Main Step in the Proof of Theorem \ref{main1}}\label{outline}
 We will combine properties of the scattering operator $\mca$ and propagation of conormal singularities for linear equations to prove the following:
\begin{theorem} \label{main-step} Let $f_j(u),$ $j=1,2,$ satisfy hypotheses {\bf H.\ref{H1}} to {\bf H.\ref{H4}} and {\bf B.\ref{H4N}}. Let $\mcl_{j\pm},$ $j=1,2,$ denote the corresponding radiation fields given by \eqref{rsladf}.  Given $\Upsilon_0\in C_0^\infty(\mr\times \ms^2),$ independent of $\omega,$ and such that $\int_{\mr} \Upsilon_0(s) ds=0,$ then according to Theorem \ref{sup-thm}, there exist  $\vphi_j,\psi_j\in C_0^\infty(\mr^3),$ $j=1,2$  such that
$\mcl_{j-} (\vphi_j,\psi_j)=\Upsilon_0,$ $j=1,2.$  Let $u_j$ be the solution to 
\begin{gather}
\begin{gathered}
\square u_j+ f_j(u_j)=0, \text{ on } \mr  \times \mr^3, \\
u_j(0,x)=\vphi_j(x), \;\ \p_t u_j(0,x)= \psi_j(x).
\end{gathered} \label{weqj}
\end{gather}
Let $\mca_j$ denote the scattering operator associated with $f_j.$  If $\mca_1=\mca_2,$  then the third and fourth derivatives of $f_1$ and $f_2$ satisfy 
\begin{gather}
f_1^{(3)}(u_1(t,x))=f_2^{(3)}(u_2(t,x)) \text{ and }  f_1^{(4)}(u_1(t,x))=f_2^{(4)}(u_2(t,x)), \label{eq-Der}
\end{gather}
 for all $(t,x)\in \mr\times \mr^3.$
\end{theorem}

The proof of this  result will take the bulk of the paper, but once it is established, we can prove Theorem \ref{main1}.

\subsection{Proof of Theorem \ref{main1}}  We will prove Theorem \ref{main1} as a consequence of Theorem \ref{main-step}.  So let us assume we have proved \eqref{eq-Der}.

\begin{proof} Let $\Upsilon_0\in C_0^\infty(\mr),$ be independent of $\omega,$  and furthermore assume that $\ds \int_\mr \Upsilon_0(s) ds=0.$ Let  $u_j(t,x)$ satisfy \eqref{weqj} with initial data $\vphi_j,\psi_j\in C^\infty(\mr^3)$ such that $\mcl_{+j}(\vphi_j,\psi_j)(s,\omega)=\Upsilon_0(s),$ $j=1,2.$   Notice that a function  $\Upsilon_0\in C_0^\infty(\mr^3)$ is such that $\ds \int_\mr \Upsilon_0(s) ds=0$ if and only if $\Upsilon_0= G'(s),$ where $G\in C_0^\infty(\mr)$  and the space of such $\Upsilon_0$ dense in $L^2(\mr).$ 

We know from  Theorem \ref{sup-thm} that the corresponding initial data $\vphi_j$ and $\psi_j$  in \eqref{weqj} are radial and $\vphi_j, \psi_j\in C_0^\infty(\mr^3),$   and from Theorem \ref{ShaStru} and the result of Struwe \cite{Str}, since the initial data is radial,  $u_j\in C^\infty(\mr\times \mr^3).$   Theorem \ref{main-step} implies that 
\begin{gather*}
\begin{gathered}
f_1^{(3)}(u_1(t,x))= f_2^{(3)}(u_2(t,x)), \text{ and }  f_1^{(4)}(u_1(t,x))= f_2^{(4)}(u_2(t,x)), \; (t,x) \in \mr\times \mr^3.
\end{gathered}%\label{f13f23}
\end{gather*}
and in particular, for $t=0,$
\begin{gather}
\begin{gathered}
f_1^{(3)}(\vphi_1(s))= f_2^{(3)}(\vphi_2(s)), \text{ and }  f_1^{(4)}(\vphi_1(s))= f_2^{(4)}(\vphi_2(s)), \text{ for all } s \in \mr.
\end{gathered}\label{f13f23}
\end{gather}

By differentiating the first equation in \eqref{f13f23} with respect to $(t,x)$ we obtain
\begin{gather}
f_1^{(4)}(\vphi_1(s)) \vphi_1'(s)= f_2^{(4)}(\vphi_2(s)) \vphi_2'(s), \text{ for all } s \in \mr,\label{f14f24}
\end{gather}
 and we conclude that 
 \begin{gather*}
 \vphi_1'(s)= \vphi_2'(s) \text{ for all  }  s \text{ such that }    f_1^{(4)}(\vphi_1(s))= f_2^{(4)}(\vphi_2(s))\not=0,
 \end{gather*}
We know from assumption {\bf B.\ref{H4N}} that  $f_1^{(4)}(\vphi_1(s))= f_2^{(4)}(\vphi_2(s))=0$ if and only if $\vphi_1(s)=\vphi_2(s)=0,$
and so we conclude that
 \begin{gather*}
 \vphi_1'(s)= \vphi_2'(s) \text{ for all  }  s \in \mco=\{s\in \mr: \vphi_1(s)\not=0\}\cap \{s\in \mr:  \vphi_2(s)\not=0\}.
 \end{gather*}
 The set $\mco$ is open and since $\vphi_1$ and $\vphi_2$ are compactly supported, 
 \begin{gather*}
 \mco= \bigcup_{j-1}^k I_j, \; I_j=(a_j,b_j),
 \end{gather*}
 and therefore
 \begin{gather*}
 \vphi_1(s)-\vphi_2(s)= c_j \text{ on } I_j, 1\leq j \leq k.
 \end{gather*}
 But since  $\vphi_1, \vphi_2\in C^\infty$ and $\vphi_1(a_j)=\vphi_2(a_j),$ it follows that $c_j=0,$ $1\leq j\leq k,$  and therefore $\vphi_1=\vphi_2=\vphi$ and
 $f_1(\vphi(s))=f_2(\vphi(s)).$

 As discussed above, the set of $\Ups_0$ is dense in $L^2(\mr)$ and so,  by continuity of the radiation fields,
 this holds for all $\Ups_0\in L^2(\mr)$ and therefore $f_1(\vphi)=f_2(\vphi)$ for all $\vphi\in C_0^\infty(\mr^3)$ and radial.
Therefore  $f_1(s)=f_2(s)$ for all $s\in \mr.$ 
 \end{proof}

\subsection{ The steps of the proof of Theorem \ref{main1}}    For $\Ups\in L^2(\mr\times \ms^2),$ we know from the work of Bahouri and G\'erard \cite{BahGer} and that there exists $u\in X(\mr;\mr^3)$ and $\vphi\in  {\dot{H}}^1(\mr^3)$ and $\psi\in L^2(\mr^3)$ such that
\begin{gather}
\begin{gathered}
\square u +f(u)=0, \\
u(0,x)=\vphi(x), \;\ \p_t u(0,x)=\psi(x), \\
\text{ such that } \mcl_-(\vphi,\psi)=\Ups.
\end{gathered}\label{u-rad}
\end{gather}

We let $\Ups_k\in L^2(\mr\times \ms^2),$ $k=0,1,2,3,4,$ and  $\eps_j>0,$ $j=1,2,3,4$   and let $\Ups$ be of the form 
\begin{gather*}
\Ups=\Ups_0+ \Ups_\veps, \text{ where } \veps=(\eps_1,\eps_2,\eps_3,\eps_4) \text{ and } \Ups_\veps=\sum_{k=1}^4\eps_j \Ups_k.
\end{gather*}
The scattering operator acting on $\Ups$ is given by
\begin{gather*}
\mcn_+ u= \mcl_+(\vphi,\psi)=\mca(\Ups_0+\eps_1 \Ups_1+\eps_2 \Ups_2+\eps_3\Ups_3+\eps_4 \Ups_4).
\end{gather*}

The proof of Theorem \ref{main-step} consists of three parts:
\begin{enumerate}[Step 1.]
\item We establish an asymptotic expansion of the scattering operator of the following form
\begin{gather*}
\begin{gathered}
\mca(\Ups_0+\eps_1 \Ups_1+\eps_2 \Ups_2+\eps_3\Ups_3+\eps_4 \Ups_4)= \\  \mca(\Ups_0)+\sum_{|\alpha|\leq 4} \vec{\eps}\,{}^\alpha \Xi_{\alpha} +O_{L^2(\mr\times \ms^2)}(|\vec{\eps}|^5), \;\
\Xi_{\alpha}\in L^2(\mr\times \ms^2).
\end{gathered}
\end{gather*} 
\item  We pick $\Ups_0\in C_0^\infty(\mr),$  as in the statement of Theorem \ref{main1}. We choose $\Ups_k,$ $k=1,2,3,4$  to be the radiation fields of suitably chosen conormal spherical waves.
\item Let $\vphi_0,\psi_0 \in C_0^\infty(\mr^3)$ be such that $\mcl(\vphi,\psi)=\Ups_0.$  Let $u_0$ be the solution to
\begin{gather}
\begin{gathered}
\square u_0+ f(u_0)=0, \\
u_0(0)=\vphi_0, \p_t u_0(0)=\psi_0, \\
\text{ and } \mcl_-(\vphi_0,\psi_0)=\Ups_0.
\end{gathered}\label{rad-u0}
\end{gather}
\item We compare the singularities of the terms $\Xi_{\alpha},$ $|\alpha|=3,$  with the singularities $\Ups_k,$ $k=1,2,3,$ and in this case we take $\Ups_4=0$.   We show that  $\Xi_{\alpha}$ contains additional conormal singularities and by computing the principal symbol of those, and by varying $\Ups_k,$ $k=1,2,3,$  that we can determine  $f^{(3)}(u_0(t,x))$  for all $(t,x)\in \mr \times \mr^3.$
\item We repeat the procedure for  from the singularities of $\Xi_\alpha$ with $|\alpha|=4$ and show that we can determine  $f^{(4)}(u_0(t,x))$  for all $(t,x)\in \mr \times \mr^3.$
 \item In particular this shows that if $f_1$ and $f_2$ have the same scattering operator,  and $u_{10}$ and $u_{20}$ are the corresponding solutions to \eqref{rad-u0}, then 
 \begin{gather*}
 f_1^{(3)}(u_{10}(t,x))= f_2^{(3)}(u_{20}(t,x)) \text{ and }  f_1^{(4)}(u_{10}(t,x))= f_2^{(4)}(u_{20}(t,x)) \text{ for all } (t,x)\in \mr\times \mr^3.
 \end{gather*}
\end{enumerate}

\section{The Linearization and Asymptotic Expansion}
As we described above, in the first step we proceed \`a la Carles and Gallagher \cite{CarGal}, linearize the equation about a $C^\infty$ solution $u_0(t,x)$ and analyze the corresponding asymptotic expansion.    
  While a complete expansion was established in \cite{CarGal}, using the fact that $f(u)$ is real analytic, here we establish an expansion up to order five.

Let $u$ satisfy $\square u=-f(u)$ and we write $u=u_0+(u-u_0),$ and so
\begin{gather}
\begin{gathered}
\square u= \square u_0+ \square(u-u_0)= -f(u_0+(u-u_0))=   -f(u_0)- f'(u_0) (u-u_0)- \\ \frac{1}{2!} f^{(2)}(u_0)(u-u_0)^2-
\frac{1}{3!} f^{(3)}(u_0)(u-u_0)^3- \frac{1}{4!} f^{(4)}(u_0)(u-u_0)^4- G(u,u_0) (u-u_0)^5,  \\ \text{ where }
G(u,u_0)= \frac{1}{4!}\int_0^1 f^{(5)} ((1-t) u_0 +t u) (1-t)^4 dt.
\end{gathered}\label{def-G}
\end{gather}
Next we write $y=(t,x)$ and 
\begin{gather}
\begin{gathered}
u-u_0= w_{\veps}+ \zed(y,\veps), \text{ where } 
w_{\veps}=\sum_{|\alpha|\leq 4} \veps\,{}^\alpha w_\alpha, 
\end{gathered}\label{defw-exp}
\end{gather}
We  substitute this expression into \eqref{def-G}, match powers of $\veps$ up to order four, and obtain bounds for the remainder. The term independent of $\veps,$  $u_0,$ satisfies the following equation:
\begin{gather}
\begin{gathered}
\square u_0=-f(u_0), \\
\mcn_- u_0= \Upsilon_0,
\end{gathered}\label{eq-u0}
\end{gather}
where $\mcn_\pm$ were defined in \eqref{radf-func}.  We use this notation here because the Cauchy data of $u_0$ has not been specified.  We know there exists a unique $u_0\in X(\mr;\mr^3)$ satisfying \eqref{eq-u0} because the maps $\mcl_\pm$ defined in \eqref{nonl-Iso} are isometries.

The terms in $\veps\,{}^\alpha$ with $|\alpha|=1$ are denoted by 
\begin{gather}
w_1=w_{1,0,0,0}, \;\ w_2=w_{0,1,0,0}, \;\ w_3=w_{0,0,1,0} \text{ and } w_4=w_{0,0,0,1}, \label{convw}
\end{gather}
 and satisfy
\begin{equation}
\begin{aligned}
& \square w_j= -f'(u_0) w_j, \\
& \mcn_- w_j= \Upsilon_j, \;\ j=1,2,3,4.
 \end{aligned}\label{defwj}
 \end{equation}
 We shall prove that given $u_0\in X(\mr,\mr^3)$ there exists a unique $w_j\in X(\mr;\mr^3)$ satisfying \eqref{defwj}.
 
 To compute the terms in $\vepa$  with $|\alpha|=2,$  we write $\alpha=\beta_1+\beta_2,$ with  $|\beta_1|=|\beta_2|=1,$ and we have
 \begin{gather}
\begin{gathered}
\square w_{\alpha}  = -f'(u_0) w_{\alpha}- \frac{1 }{2!} f^{(2)}(u_0)\sum_{\alpha=\beta_1+\beta_2, |\beta_1|=|\beta_2|=1}  w_{\beta_1} w_{\beta_2}, \\\
 \text{ and } \mcn_- w_{\alpha}=0.
\end{gathered}\label{defwjk}
\end{gather}
We shall prove that, given $u_0, w_{\beta_1}, w_{\beta_2}\in X(\mr;\mr^3),$ $|\beta_1|=|\beta_2|=1,$ there exists a unique $w_\alpha \in X(\mr;\mr^3)$ satisfying \eqref{defwjk}.

To compute the terms of order three, we write
\begin{gather*}
\alpha=\beta_1+\beta_2,  \;\ |\beta_1|=1, \ |\beta_2|=2 \text{ and }
\alpha=\gamma_1+\gamma_2+\gamma_3, \;  |\gamma_j|=1, \; j=1,2,3.
\end{gather*}
 and we find
\begin{gather}
\begin{gathered}
\square w_{\alpha}= -f'(u_0) w_{\alpha}-\frac{1 }{2!}  f^{(2)}(u_0)\sum_{\alpha=\beta_1+\beta_2,|\beta_1|=1, |\beta_2=2}  w_{\beta_1} w_{\beta_2} -\frac{ 1 }{3!} f^{(3)}(u_0) \sum_{\alpha=\gamma_1+\gamma_2+\gamma_3, |\gamma_j|=1}w_{\gamma_1}w_{\gamma_2}w_{\gamma_3}, \\
 \mcn_- w_{\alpha}=0.
\end{gathered}\label{defw123}
\end{gather}
As in the first two cases, we shall prove that, given $u_0, w_{\beta_1}, w_{\beta_2}, w_\alpha\in X(\mr;\mr^3),$ with $|\beta_1|=|\beta_2|=1,$ and $|\alpha|=2,$  there exists a unique $w_\alpha \in X(\mr;\mr^3)$ satisfying \eqref{defw123}.

We split the terms with $|\alpha|=4$ into
\begin{gather*}
\alpha=\beta_1+\beta_2, \;\ |\beta_1|=2, \ |\beta_2|=2  \text{ or } |\beta_1|=1, |\beta_2|=3, \\
\alpha=\gamma_1+\gamma_2+\gamma_3, \;  |\gamma_j|=1, \; j=1,2, \; |\gamma_3|=2, \\
\alpha= \zeta_1+\zeta_2+\zeta_3+\zeta_4, \;\ |\zeta_j|=1, j=1,2,3,4,
\end{gather*}
and write
\begin{gather}
\begin{gathered}
\square w_{\alpha}= -f'(u_0) w_{\alpha}- \frac{ 1 }{2!}  f^{(2)}(u_0)\sum_{|\beta_1|=1, |\beta_2|=3} w_{\beta_1} w_{\beta_2}
-\frac{1}{2!}  f^{(2)}(u_0)\sum_{|\beta_1|=2, |\beta_2|=2} w_{\beta_1} w_{\beta_2}+ \\
\frac{1}{3!}  f^{(3)}(u_0)\sum_{|\gamma_1|=1, |\gamma_2|=1, |\gamma_3|=2} w_{\ga_1} w_{\ga_2}w_{\ga_3}+ 
\frac{1}{4!}  f^{(4)}(u_0)\sum_{|\zeta_j|=1, j=1,2,3,4} w_{\zeta_1}w_{\zeta_2}w_{\zeta_3}w_{\zeta_4}, \\
 \mcn_- w_{\alpha}=0.
\end{gathered}\label{defw1234}
\end{gather}

Again, we shall prove that, given $u_0, w_{\alpha}\in X(\mr;\mr^3),$  $|\alpha|=1,2,3,$  there exists a unique $w_\alpha \in X(\mr;\mr^3)$ with $|\alpha|=4$ satisfying \eqref{defw1234}.

We then estimate the remainder $\zed(y,\veps)$ given by \eqref{defw-exp}. We define
\begin{gather*}
 w_{\veps}= w_{1,\veps}+w_{2,\veps}+w_{3,\veps}+w_{4,\veps}, \text{ with } w_{j,\veps} \text{ homogeneous of degree j  in } \veps.
 \end{gather*}
 Using this notation we find that $\zed(y,\veps)$ satisfies
\begin{gather}
\begin{gathered}
\square \zed = -V_1 \zed- V_2 \zed^2- V_3 \zed^3- V_4 \zed^4- V_4 \zed^5- R, \\ \mcn_- \zed=0, \\ \text{ where } \\
V_1= f'(u_0)+  f''(u_0) w_{\veps}+ \frac{1}{2} f^{(3)}(u_0) w_{\veps}^2+ \frac{1}{3!} w_{\veps}^3 + 5 Gw_{\veps}^4, \\
V_2= \ha f''(u_0)+ \frac{1}{2} f^{(3)}(u_0) w_{\veps}+ \frac{1}{4} f^{(4)}(u_0)w_{\veps}^2+ 10 G w_{\veps}^3, \\
V_3= \frac{1}{3!} f^{(3)}(u_0)+ \frac{1}{3!} f^{(4)}(u_0) w_{\veps}+ 10 G w_{\veps}^2, \\
V_4= \frac{1}{4!} f^{(4)}(u_0) + 5 G w_{\veps}, \\
V_5= G, \\
R= \frac{1}{2!}f^{(2)}(u_0) R_1 +\frac{1}{3!}f^{(3)}(u_0) R_2 +\frac{1}{4!}f^{(4)}(u_0)R_3+ R_4, \\
R_1=  w_{4,\veps}^2+2 w_{4,\veps} A +w_{3,\veps}^2+2 w_{2,\veps}w_{3,\veps}, \;\ A= w_{1,\veps}+w_{2,\veps}+w_{3,\veps}, \\
R_2= 3 w_{1,\veps} \, q^2  + 3 w_{1,\veps}^2(w_{3,\veps}+w_{4,\veps})+ q^3, \;\ q=w_{2,\veps}+w_{3,\veps}+w_{4,\veps}\\
R_3= 4 w_{1,\veps}\, q^3 + 6w_{1,\veps}^2 \, q^2
+ 4 w_{1,\veps}^3\, q+ q^4, \\
R_4=  G w_{\veps}^5.
\end{gathered}\label{eqw4}
\end{gather}

The following result establishes the existsnce of $w_\alpha,$ $|\alpha|\leq 4$ and an estimate for the remainder $\mcz(y,\veps):$
\begin{theorem}\label{expansion1} Let $\Upsilon_j \in L^2(\mr\times \ms^2),$ $j=0,1,2,3,4,$  and let $u$  be the unique Shatah-Struwe solution of
\begin{gather}
\begin{gathered}
\square u+ f(u)=0, \\
\mcn_- u= \Upsilon_0+\eps_1 \Upsilon_1+\eps_2\Ups_2+\eps_3\Ups_3+\eps_4\Ups_4
\end{gathered}\label{u-ups}
\end{gather}
and let $u_0$ be the unique Shatah-Struwe solutions of \eqref{eq-u0}. Then there exist unique $w_\alpha\in X(\mr; \mr^3),$ $|\alpha|\leq 4,$  which satisfy  equations \eqref{defwj} to \eqref{defw1234}.
Let $w_{\veps}$ be defined by \eqref{defw-exp}, corresponding to these $w_\alpha,$ and let 
\begin{gather}
\mcz(y,\veps)=u(y)-\left( u_0(y)+ w_{\veps}(y) \right). \label{mce-def}
\end{gather}
Then there exists $\eps_0>0$ and $C>0$ such that for $\eps_j<\eps_0,$ $j=1,2,3,4$ 
\begin{gather}
\begin{gathered}
||\mcz||_{L^5((-\infty, T);L^{10}(\mr^3))}\leq C |\veps|^5 \text{ and }  ||\square \mcz||_{L^1((-\infty, T);L^{2}(\mr^3))}\leq C |\veps|^5, \text{ for every } T\in \mr.
\end{gathered}\label{mcet0}
\end{gather}

Moreover, for $|\vec{\eps}|<\eps_0,$
\begin{gather}
\begin{gathered}
\mcn_+(u_0+ w_{\veps}(t,x))- \mca(\Upsilon_0+\sum_{j=1}^4 \eps_j \Upsilon_j)= O_{L^2(\mr\times \ms^2)}(|\veps|^5).
\end{gathered}\label{radf-approx}
\end{gather}
\end{theorem}

This implies that
\begin{corollary} Let $\Ups_j\in L^2(\mr\times \ms^2),$ $j=0,1,2,3,4,$ let $u_0$ satisfy \eqref{eq-u0} and let
 $w_\alpha,$ $|\alpha|\leq 4,$ satisfy \eqref{defwj} to \eqref{defw1234}.  Let $\eps_0$ be as in Theorem \ref{expansion1}. Then, for $|\veps|<\eps_0,$
\begin{gather}
\begin{gathered}
\mcl_+\mcl_-^{-1}(\Upsilon_0+\sum_{j=1}^4 \eps_j\Upsilon_j)=   \mca(\Upsilon_0+\sum_{j=1}^4 \eps_j\Upsilon_j)= 
\mca (\Upsilon_0)+ \sum_{|\alpha|\leq 4} \veps\,{}^\alpha \Xi_\alpha+  O_{L^2(\mr\times \ms^2)}(|\veps|^5), \\
\text{ where } \Xi_\alpha=\mcn_+ w_\alpha, \text{ and we adopt the notation \eqref{convw}.} 
\end{gathered}\label{radfeps}
\end{gather}
\end{corollary}

%%%%%%%%%%%%
\section{Proof of Theorem \ref{expansion1}}

Since the maps $\mcl_\pm$ defined in \eqref{nonl-Iso} are isomorphisms, we know that there exist  unique $u, u_0 \in X(\mr;\mr^3)$ satisfying equations \eqref{u-ups} and \eqref{eq-u0}, we only need to prove that  there exist unique $w_\alpha \in X(\mr;\mr^3),$ $|\alpha|=1,2,3,4,$  satisfying \eqref{defwj}, \eqref{defwjk}, \eqref{defw123} and \eqref{defw1234} and if  $\zed(y,\veps)$ is given by \eqref{mce-def}, it satisfies \eqref{mcet0}, and  finally \eqref{radf-approx} holds.

  We will prove the following two propositions:
\begin{prop}\label{sol123} Let $u_0\in X(\mr;\mr^3)$ satisfy \eqref{eq-u0} and let $\Ups_1\in L^2(\mr\times \ms^2).$  Then
there exist unique $w_\alpha \in X(\mr;\mr^3),$ $|\alpha|=1,$ satisfying \eqref{defwj}. Furthermore, 
\begin{enumerate}[1.]
\item Given  $u_0, w_{\alpha}\in X(\mr,\mr^3),$ $|\alpha|=1,$ then for $|\alpha|=2,$ there exists a unique $w_{\alpha}\in X(\mr;\mr^3),$ satisfying \eqref{defwjk}, 
\item Given  $u_0, w_{\alpha}\in X(\mr,\mr^3),$ $|\alpha|=1, 2,$  then for $|\alpha|=3,$ there exists a unique $w_{\alpha}\in X(\mr;\mr^3),$ satisfying \eqref{defw123}, 
\item Given  $u_0, w_{\alpha}\in X(\mr,\mr^3),$ $|\alpha|=1,2,3,$  then for $|\alpha|=4,$ there exists a unique $w_{\alpha}\in X(\mr;\mr^3),$ satisfying \eqref{defw1234}.
\end{enumerate}
\end{prop}
and
\begin{prop}\label{solw4} Let $u \in X(\mr;\mr^3)$ be the unique Shatah-Struwe solution of \eqref{u-ups}.   Let  
$u_0 \in X(\mr;\mr^3)$ satisfy \eqref{defwj} and let $w_\alpha\in X(\mr;\mr^3),$ $|\alpha|\leq 4,$  satisfy \eqref{defwj} to \eqref{defw1234}.  Let $G$ be defined by \eqref{def-G} and let $\zed(y,\veps)$ be defined by \eqref{defw-exp}.  Then there exists $C>0$ and $\eps_0>0$ such that for $|\vec{\eps}|<\eps_0,$ and for any $T\in \mr,$ equation \eqref{mcet0} is satisfied.
\end{prop}
Assume these Propositions have been proved. We can then prove Theorem \ref{expansion1}.
\begin{proof}
We know that $\mcn_-\mcz(y,\veps)=0$ and that $\mcz(y,\veps)$ and $\mcg=\square \mcz$ satisfy \eqref{mcet0}.  Then it follows that, for any $T\in \mr,$
\begin{gather*}
\square \mcz =  \mcg, \;\ \mcg \in  L^1((-\infty,T);L^2(\mr^3)),\;\ 
 ||\mcg||_{L^1((-\infty,T);L^2(\mr^3))}\leq C |\veps|^5,\\
\mcr_-(\mcz(T), \p_t \mcz(T), -\mcg)=0.
\end{gather*}
This implies that
\begin{gather*}
\mcr_-(\mcz(T),\p_t \mcz(T),0)= \mcr_-(0,0,\mcg)(s-T,\omega),
\end{gather*}
and so we deduce from  \eqref{liniso} and \eqref{est-radf} that
\begin{gather*}
E_0(\mcz(T), \p_t \mcz(T))= ||\mcr_-(\mcz(T),\p_t \mcz(T),0)||_{L^2(\mr\times \ms^2)}=\\ ||\mcr_-(0,0,\mcg)(s-T,\omega)||_{L^2(\mr\times \ms^2)}\leq ||\mcg||_{L^1((-\infty,T); L^2(\mr^3))}.
\end{gather*}
So it follows that
\begin{gather*}
E_0(\mcz(T), \p_t \mcz(T))\leq ||\mcg||_{L^1((-\infty,T); L^2(\mr^3))} \leq C |\veps|^5.
\end{gather*}

But \eqref{mcet0} holds for all $T,$ and therefore $||\mcg||_{L^1((T,\infty);L^2(\mr^3))}\leq C |\veps|^5,$ and now the same estimate, applied to the forward radiation field, shows that
\begin{gather*}
||\mcn_+ \mcz||_{L^2(\mr\times \ms^2)}=|| \mcr_+(\mcz(T), \p_t \mcz(T), -\mcg)|| _{L^2(\mr\times \ms^2)}
\leq \\ E_0(\mcz(T), \p_t \mcz(T)) + ||\mcg||_{L^1([T,\infty); L^2(\mr^3)}\leq C |\veps|^5.
\end{gather*}
This is the same as saying that 
\begin{gather*}
\mcn_+(u-(u_0+w_{\veps})) = O_{L^2(\mr\times \ms^2)}(|\veps|^5).
\end{gather*}
This proves \eqref{radfeps} and  hence it proves Theorem \ref{expansion1}.
\end{proof}
\subsection{The proof of Proposition \ref{sol123}}

We begin by proving the following:
\begin{lemma}\label{StPot} If $V\in L^{\frac54}(\mr; L^{\frac52}(\mr^3)),$  $g\in L^1(\mr; L^2(\mr^3))$ and let $F\in L^2(\mr\times \ms^2),$ then there exists a unique $w\in X(\mr;\mr^3)$ such that
\begin{gather}
\begin{gathered}
\square w= Vw+g, \text{ in } \mr \times \mr^3, \\
\mcn_- w=F,
\end{gathered}\label{StPot1}
\end{gather}
where $X(\mr;\mr^3)$  is the space defined in \eqref{defX}.
\end{lemma}
\begin{proof} This is a linear equation, but with a potential in $L^{\frac54}(\mr; L^{\frac52}(\mr^3))$ and a forcing term  $g\in L^1(\mr; L^2(\mr^3)),$ and a given radiation field. So this result is by no means obvious. We follow the strategies used by Bahouri and G\'erard \cite{BahGer} to prove asymptotic completeness for \eqref{Weq} and by Carles and Gallagher \cite{CarGal} to establish the analogue of \eqref{radfeps} in the real analytic case.  Let $w_0\in X(\mr;\mr^3)$ be the solution to
\begin{gather}
\begin{gathered}
\square w_0=g, \\
w_0(0)=\vphi, \; \p_t w_0(0)=\psi,
\end{gathered}\label{eqw0}
\end{gather}
 $(\vphi_0,\psi_0)\in \hone(\mr^3) \times L^2(\mr^3)$ such that $\mcr_-(\vphi_0,\psi_0)=F.$ It follows from Theorem \ref{ShaStru} that $w_0\in X(\mr;\mr^3).$  Since $V\in L^{\frac54}(\mr; L^{\frac52}(\mr^3)),$ and $w_0 \in  L^{5}(\mr; L^{10}(\mr^3)),$ given $\del>0,$ there exist $T_0<0$ and $T_1>0$ such that
\begin{gather}
\begin{gathered}
||V||_{L^{\frac54}((-\infty,T_0]; L^{\frac52}(\mr^3))}<\del, \;\ ||V||_{L^{\frac54}([T_1,\infty); L^{\frac52}(\mr^3))}<\del. \;\ 
\end{gathered}\label{estVw}
\end{gather}

Since we are dealing with the backward radiation field in \eqref{StPot1}, we use the first inequality in \eqref{estVw}.
For any $ w\in {L^{5}((-\infty,T_0]; L^{10}(\mr^3))},$ let $\upsilon$ satisfy
\begin{gather}
\begin{gathered}
\square \upsilon= V (w_{0}+ w) \text{ in } (-\infty, T_0) \times \mr^3\\
\upsilon(T_0)= \vphi_{w}, \;\  \p_t \upsilon(T_0)= \psi_w, \text{ such that } \\
\mcr_-(\vphi_w,\psi_w, V (w_0+w))(s+T_0)=0.
\end{gathered}\label{weq-it1}
\end{gather}
 We want to show that the map
\begin{gather}
\begin{gathered}
\mcc: L^{5}((-\infty,T_0]; L^{10}(\mr^3)) \longmapsto L^{5}((-\infty,T_0]; L^{10}(\mr^3)), \\
w \longmapsto \ups
\end{gathered}\label{mcc0}
\end{gather}
is bounded and a contraction for small $\del.$  The fact that the equation is linear, allows us to work in the entire space
$L^{5}((-\infty,T_0]; L^{10}(\mr^3))$ instead of in a ball of small radius as in \cite{BahGer}.
It follows from \eqref{STES}  that 
\begin{gather*}
||\upsilon||_{L^{5}((-\infty,T_0]; L^{10}(\mr^3))}\leq C_{10}\left( ||\nabla\vphi_w||_{L^2} + ||\psi_w||_{L^2}+ || V(w_0+w)||_{L^1((-\infty,T_0]; L^2(\mr^3))}\right).
\end{gather*}
By assumption, $\mcr_-(\vphi_w,\psi_w)=\mcr(0,0,- V (w+ w_0)(s+T_0))=0,$ but we also know that
\begin{gather}
\begin{gathered}
||\mcr_-(\vphi_w,\psi_w,0)||_{L^2}= ||\nabla \vphi_w||_{L^2} + ||\psi_w||_{L^2},\\
 \text{ and it follows from \eqref{est-radf}  that }\\
 ||\mcr_-(0,0,- V (w+ w_0))(s+T_0)||_{L^2(\mr\times \ms^2)}\leq ||V (w+ w_0)||_{L^1((-\infty,T_0]; L^2(\mr^3))}.
 \end{gathered}\label{hypV}
  \end{gather}
Therefore, 
\begin{gather}
||\ups||_{L^{5}((-\infty,T_0]; L^{10}(\mr^3))} \leq 2C_{10} || V(w_0+w)||_{L^1((-\infty,T_0]; L^2(\mr^3))}. \label{estwj}
\end{gather}
But by applying H\"older inequality twice with $p=5$ and $q=\frac54$ we obtain 
\begin{gather}
\begin{gathered}
|| V(w_0+w)||_{L^1((-\infty,T_0]; L^2(\mr^3))}=\int_{-\infty}^{T_0}\left[ \int_{\mr^3} |V|^2 |w_0+w|^2 dx\right]^\ha dt\leq \\
\int_{-\infty}^{T_0} \left[\int_{\mr^3} |V|^{\frac52} dx\right]^{\frac25}
 \left[\int_{\mr^3} |w_0+ w|^{10} dx\right]^{\frac{1}{10}} dt \leq \\
\left[ \int_{-\infty}^{T_0} \left[\int_{\mr^3} |V|^{\frac52} dx\right]^{\frac12} dt\right]^{\frac45}
 \left[ \int_{-\infty}^{T_0}\left[\int_{\mr^3} |w_0+ w|^{10} dx\right]^{\frac{1}{2}} dt\right]^{\frac15}=\\
 ||V||_{L^{\frac54}((-\infty,T_0]; L^{\frac52}(\mr^3))} ||w_0+ w||_{L^5((-\infty,T_0]; L^{10}(\mr^3))}.
\end{gathered}\label{firt-H}
\end{gather}
It follows from \eqref{estVw} and \eqref{firt-H} that 
\begin{gather}
||\upsilon||_{L^{5}([T_0,\infty); L^{10}(\mr^3))} \leq 2C_{10}\del ||w+w_0||. \label{energy-ups}
\end{gather}
So  the map $\mcc$ defined in  \eqref{mcc0} is bounded.  

 Next observe that if $\upsilon_j=\mcc w_j,$ $j=1,2,$ then, since equation  \eqref{weq-it1} is linear, using the same argument it follows that 
\begin{gather*}
||\upsilon_1-\upsilon_2||_{L^{5}([T_0,\infty); L^{10}(\mr^3))}  \leq 2 C_{10}\del ||w_1-w_2||_{L^{5}([T_0,\infty); L^{10}(\mr^3))},
\end{gather*}
and so, if $2C_{10}\del<1,$  $\mcc$ is a contraction, and its fixed point $w^*$ satisfies 
\begin{gather*}
\square (w^*+w_0)= V(w^*+w_0)+g, \text{ in } (-\infty,T_0] \times \mr^3, \\
\mcn_- (w^*+w_0)(s)=F(s+T_0,\omega),
\end{gather*}

Now, we  want to extend $w^*$ to a global solution on $\mr\times \mr^3.$ We take a partition of $[T_0,T_1],$ where $T_0,T_1$ are such that \eqref{estVw} holds,  consisting  of non-overlapping intervals $I_j=[a_j,a_{j+1}],$ $0\leq j \leq N,$ such that $a_0=T_0,$ $a_N=T_1$ and 
\begin{gather}
\begin{gathered}
||V||_{L^{\frac54}(I_j; L^{\frac52}(\mr^3))}<\del, \;\  ||w_0||_{L^{5}(I_j; L^{10}(\mr^3))}<\del. \\ 
\end{gathered}\label{estVwj}
\end{gather}
We claim there exist unique $W_j\in X([a_0,a_1], \mr^3),$  $j=0,1,\ldots, N,$  such that
\
\begin{gather}
\begin{gathered}
\square W_0= V (W_0+w_0), \text{ in } (a_0, a_1) \times \mr^3,\\
W_0(T_0)= w^*(T_0), \;\  \p_t W_0(T_0)=w^*(T_0),
\end{gathered}\label{wi0}
\end{gather}
and for $1\leq j \leq N,$ 
\begin{gather}
\begin{gathered}
\square W_j= V (W_j+w_0), \text{ in } (a_j, a_{j+1}) \times \mr^3,\\
W_j(a_j)=W_{j-1}(a_j), \;\  \p_t W_j(a_j)=\p_t W_{j-1}(a_j).
\end{gathered}\label{defWj}
\end{gather}

We proceed as above, and the key point here is that the constant $C_{10}$ in the Strichartz estimate \eqref{STES} does not depend of the size of the interval. 

If  $\vtheta\in L^5(I_0; L^{10}(\mr^3)),$  and for such $\vtheta,$ let $\ups= \mcc_{I_0} \vtheta$ denote the solution to 
\begin{gather*}
\square \ups= V(\vtheta+w_0),\\
\ups (T_0)= w^*(T_0),  \;\ \p_t \ups (T_0)= w^*(T_0).
\end{gather*}
We proceed exactly as above and use H\"older inequality and \eqref{STES} to conclude that
\begin{gather*}
||\ups||_{L^5(I_0; L^{10}(\mr^3))} \leq C_{10}\left( ||\nabla w^*(T_0)||_{L^2}+ ||\p_t w^*(T_0)||_{L^2}+ 
||V||_{L^{\frac54}(I_0; L^{\frac52}(\mr^3))} ||\upsilon+w_0||_{L^{5}(I_0; L^{10}(\mr^3)))}\right),
\end{gather*}
so
\begin{gather*}
\mcc_{I_0}:  L^5(I_0; L^{10}(\mr^3)) \longmapsto L^5(I_0; L^{10}(\mr^3)),
\end{gather*}
and 
\begin{gather*}
||\mcc_{I_0}(\ups_1-\ups_2)||\leq C_{10} ||V||_{L^{\frac54}(I_0; L^{\frac52}(\mr^3))} ||\ups_1-\ups_2||_{L^{5}(I_0; L^{10}(\mr^3)))}\leq C_{10}\del  ||\ups_1-\ups_2||_{L^{5}(I_0; L^{10}(\mr^3)))}.
\end{gather*}
So, for $\del$ chosen as above, $\mcc_{I_0}$ is a contraction and its fixed point $W_0$ satisfies \eqref{wi0}. The same argument, with the same choice of $\del,$ applies to all intervals $I_j.$  This gives an extension of the solution $w,$ from $(-\infty,T_0]$ to $(-\infty,T_1].$ 

Next we want to show there exists a unique solution $W\in X([T_1,\infty);\mr^3)$ of 
\begin{gather*}
\square W= V (W+ w_0), \text{ in } (T_1, \infty) \times \mr^3,\\
W(T_1)= W_N(T_1), \;\   \p_t W(T_1)=\p_tW_N(T_1)
\end{gather*}
We use the same method as above, and the assumption on the norm of $V$ in $[T_1,\infty)$ in \eqref{estVw} and the choice of $\del$ guarantees that the corresponding map is a contraction. 
The function $w$ given by
\begin{gather*}
w=w^*, \; t\leq T_0,\\
w=W_j, \; t\in (a_j, a_j+1], \; j=0,1,\ldots, N,\\
w=W, \; t> T_1,
\end{gather*}
satisfies
\begin{gather*}
\square(w+w_0)= V(w+w_0)+g, \\
\mcn_-(w+w_0)=F.
\end{gather*}
This ends the proof of the Lemma.
\end{proof}

Next we will apply Lemma \ref{StPot} to prove Proposition  \ref{sol123}, but to do that,  first we need to prove estimates, 
 which control $L^pL^q$ norms of products of functions:
\begin{lemma}\label{several-ineq}  If $I \subset \mr$ is an interval, if $v_1,v_2, v_3,v_4 \in L^5(I; L^{10}(\mr^3))$ and if
$Q(x_1,x_2,x_3,x_4)$ is a homogeneous polynomial with positive coefficients of degree $j$ with $1\leq j \leq 5,$ then
\begin{gather}
\begin{gathered}
Q(v_1,v_2,v_3,v_4) \in L^{\frac{5}{j}}(I; L^{\frac{10}{j}}(\mr^3)) \text{ and } 
||Q(v_1,v_2,v_3,v_4)||_{L^{\frac{5}{j}}(I; L^{\frac{10}{j}}(\mr^3))} \leq Q(||v_1||,||v_2||,||v_3||,||v_4||), \\
\text{ where } ||v_j||=||v_j||_{L^5(I; L^{10}(\mr^3))}.
 \end{gathered} \label{vtpj}
\end{gather}
In particular, if $f\in C^5(\mr)$ is such that  $|f^{(m)}(u)|\leq C_m |u|^{5-m},$  if $u\in L^5(I; L^{10}(\mr^3))$ and $m\leq 4,$
\begin{gather}
\begin{gathered}
f^{(m)}(u) \in L^{\frac{5}{5-m}}(I; L^{\frac{10}{5-m}}(\mr^3)) \text{ and } 
||f^{(m)}(u)||_{L^{\frac{5}{5-m}}(I; L^{\frac{10}{5-m}}(\mr^3))}\leq C ||u||_{L^5(I; L^{10}(\mr^3))}^{5-m}.
\end{gathered}\label{growthf}
\end{gather}
\end{lemma}
\begin{proof}  The statements is obvious when  $Q$ is of degree 1. When $Q$ is of degree  $2,$ the Cauchy-Schwarz inequality gives
\begin{gather*}
||v_1 v_2||_{L^{\frac{5}{2}}(I; L^{5}(\mr^3))}^{\frac{5}{2}} =  
\int_{I}  \left[ \int_{\mr^3} |v_1 v_2|^{5} dx\right]^{\frac{1}{2}} dt\leq 
\int_{I} \left[\int_{\mr^3} |v_1|^{10} dx\right]^{\frac14} \left[\int_{\mr^3} |v_2|^{10} dx\right]^{\frac14} dt\leq \\
\left[ \int_{I} \left[\int_{\mr^3} |v_1|^{10} dx\right]^{\frac12} dt\right]^\ha \left[ \int_{I} \left[\int_{\mr^3} |v_2|^{10} dx\right]^{\frac12} dt\right]^\ha= ||v_1||_{L^5(I; L^{10}(\mr^3))}^{\frac52} ||v_2||_{L^5(I; L^{10}(\mr^3))}^{\frac52}.
\end{gather*}
When $j=3,$ H\"older inequality with $p=3$ gives
\begin{gather*}
||v_1 v_2 v_3||_{L^{\frac{5}{3}}(I; L^{\frac{10}3}(\mr^3))}^{\frac{5}{3}} =  
\int_{I}  \left[ \int_{\mr^3} |v_1 v_2 v_3|^{\frac{10}3} dx\right]^{\frac{1}{2}} dt\leq 
\int_{I} \left[\int_{\mr^3} |v_1|^{10} dx\right]^{\frac16} \left[\int_{\mr^3} |v_2 v_3|^{5} dx\right]^{\frac13} dt\leq \\
\left[ \int_{I} \left[\int_{\mr^3} |v_1|^{10} dx\right]^{\frac12} dt\right]^{\frac13}
 \left[ \int_{I} \left[\int_{\mr^3} |v_2v_3|^{5} dx\right]^{\frac12} dt\right]^{\frac23}=  ||v_1||_{L^5(I; L^{10}(\mr^3))}^{\frac53} ||v_2v_3||_{L^{\frac{5}{2}}(I; L^{5}(\mr^3))}^{\frac53},
% 
%  \left[ \int_{I} \left[\int_{\mr^3} |v_2|^{10} dx\right]^{\frac14}  \left[\int_{\mr^3} |v_3|^{10} dx\right]^{\frac14}  dt\right]^{\frac23}\leq \\
% ||v_1||_{L^5(I; L^{10}(\mr^3))}^{\frac53}  ||v_2||_{L^5(I; L^{10}(\mr^3))}^{\frac53}  ||v_3||_{L^5(I; L^{10}(\mr^3))}^{\frac53}  
\end{gather*}
and the result follows from the previous inequality. 
The proofs for $j=4$ and $j=5$ are very similar.
\end{proof}

We also need the following
\begin{lemma}\label{prodvw1}  Let $V_m\in L^{\frac{5}{5-m}}(\mr;L^{\frac{10}{5-m}}(\mr^3)),$ $2\leq m \leq 4,$ and let 
$w_k \in L^{5}(\mr;L^{10}(\mr^3)),$ $1\leq k \leq 4,$ then
\begin{gather}
\begin{gathered}
||V_2 w_1 w_2||_{L^{1}(\mr;L^{2}(\mr^3))} \leq ||V_2||_{L^{\frac53}(\mr;L^{\frac{10}{3}}(\mr^3))} ||w_1||_{L^5(\mr;L^{10}(\mr^3))}  ||w_2||_{L^5(\mr;L^{10}(\mr^3))}, \\
||V_3 w_1 w_2 w_3||_{L^{1}(\mr;L^{2}(\mr^3))}\leq ||V_3||_{L^{\frac52}(\mr;L^{5}(\mr^3))} 
||w_1||_{L^5(\mr;L^{10}(\mr^3))}||w_2||_{L^5(\mr;L^{10}(\mr^3))} ||w_2||_{L^5(\mr;L^{10}(\mr^3))}, \\
||V_4 w_j w_k w_l w_m||_{L^{1}(\mr;L^{2}(\mr^3))}\leq \\ ||V_4||_{L^{5}(\mr;L^{10}(\mr^3))}
 ||w_j||_{L^5(\mr;L^{10}(\mr^3))}||w_k||_{L^5(\mr;L^{10}(\mr^3))}||w_l||_{L^5(\mr;L^{10}(\mr^3))}||w_m||_{L^5(\mr;L^{10}(\mr^3))}.
\end{gathered}\label{products}
\end{gather}
\end{lemma}
\begin{proof}  We just need to apply H\"older inequality with $p=\frac54$ and $q=5$ twice to verify that
\begin{gather*}
||V_4 w_j w_k w_l w_m||_{L^{1}(\mr;L^{2}(\mr^3))} \leq
 ||V_4||_{L^{5}(\mr;L^{10}(\mr^3))}||w_j w_k w_l w_m||_{L^{\frac54}(\mr;L^{\frac52}(\mr^3))}
 \end{gather*}
 and the last inequality follows from Lemma \ref{several-ineq}. 
 
 Similarly, an application of  H\"older inequality with $p=\frac52$ and $q=\frac53$ gives
\begin{gather*}
||V_3 w_1 w_2 w_3||_{L^{1}(\mr;L^{2}(\mr^3))}\leq 
 ||V_3||_{L^{\frac52}(\mr; L^5(\mr^3))}  ||w_1w_2w_3||_{L^{\frac53}(\mr; L^{\frac{10}{3}}(\mr^3))},
\end{gather*}
and again the desired inequality follows from Lemma \ref{several-ineq}.  The same argument gives the first inequality in \eqref{products}.
\end{proof}

We know from \eqref{growthf}  that $f'(u_0) \in L^{\frac{5}{4}}(I; L^{\frac{5}{2}}(\mr^3))$ and  therefore 
we can apply Lemma \ref{StPot} to conclude that there exist unique $w_j\in L^5(\mr;L^{10}(\mr^3)),$ $j=1,2,3,4,$ which satisfy \eqref{defwj}.

Similarly, to prove the existence of $w_\alpha$ with $|\alpha|=2,$  we use Lemma \ref{several-ineq} and Lemma \ref{prodvw1} to show that 
\begin{gather*}
f^{(2)}(u_0) w_j^2, \; f^{(2)}(u_0) w_j w_k  \in L^1(\mr; L^{2}(\mr^3)),
\end{gather*}
and so Lemma \ref{StPot} implies that there exist unique $w_\alpha \in L^5((\mr);L^{10}(\mr^3)),$ $|\alpha|=2$ satisfying \eqref{defwjk}.  
 
  Now, to prove the existence of $w_\alpha$ with $|\alpha|=3,$  we need to consider products  of the type 
  $f''(u_0) w_j w_\alpha,$ $|\alpha|=2$  and $f^{(3)}(u_0) w_j w_k w_m,$ but again Lemma \ref{several-ineq} and Lemma \ref{prodvw1} show that both are in 
  $L^1(\mr; L^{2}(\mr^3)),$  and so there exist $w_\alpha\in L^5(\mr;L^{2}(\mr^3)),$ $|\alpha|=3$ satisfying \eqref{defw123}.  To show that there exists a unique $w_\alpha$  with $|\alpha|=4$ we need to show that the terms
 $f^{(2)}(u_0) w_j w_{\beta},$ $|\beta|=3,$  $f^{(2)}(u_0) w_\gamma w_{\beta},$ $|\gamma|=|\beta|=2,$  
  $f^{(3)}(u_0) w_j w_k w_{\beta}$ with $|\beta|=2$  and  $f^{(4)}(u_0) w_j w_k w_l w_m$  are in $L^1(\mr; L^{2}(\mr^3)),$ and once again this is guaranteed by Lemma \ref{several-ineq} and Lemma \ref{prodvw1}.   This ends the construction of $w_{\alpha}$ with $|\alpha|\leq 4$ and the proof of Proposition \ref{sol123}.
 
 \subsection{The Proof of Proposition \ref{solw4}}

We  know from the work of Bahouri and G\'erard \cite{BahGer} that there exist  unique $u\in X(\mr;\mr^3)$ and $u_0\in X(\mr;\mr^3)$ satisfying \eqref{u-rad} and \eqref{rad-u0}.  We also know from Proposition \ref{sol123} that there exist unique  $w_{\alpha}\in X(\mr;\mr^3),$ with $|\alpha|\leq 4$  satisfying \eqref{defwj} to \eqref{defw1234}. Let  
$\zed(y,\veps)$ be given by \eqref{defw-exp}.   We emphasize we are not proving the existence or uniqueness of $\mcz(y,\veps).$  We already know this. We only need to prove the estimates \eqref{mcet0}. 
\begin{proof} We use $\zed=\zed(y,\veps)$ and rewrite \eqref{eqw4} as
\begin{gather}
\begin{gathered}
\square \zed= \La(\zed)-  R, \\
\mcn_- \zed=0, \\
\text{ where } \La(\zed) = -V_1 \zed- V_2 \zed^2- V_3 \zed^3 - V_4 \zed^4- V_5 \zed^5, \text{ and } \\
 V_j,  \;\ 1\leq j \leq 5, \text{ and } R \text{ are given by \eqref{eqw4}. }
\end{gathered}\label{defLa}
\end{gather}
 We know from Proposition \ref{sol123} and Lemma \ref{several-ineq}  that
\begin{gather}
||R||_{L^1(\mr;L^2(\mr^3))} \leq C|\veps|^5. \label{estR}
\end{gather}

 We already know from Lemma \ref{prodvw1} that $V_j\in L^{\frac{5}{5-j}}(\mr, L^{\frac{10}{5-j}}(\mr^3)),$ where
 $V_j,$ $1\leq j \leq 4,$ is defined in \eqref{eqw4}.  Since $|f^{(5)}(u)|\leq C,$  then  $|V_5|\leq C.$
 The following is a consequence of Lemma \ref{prodvw1}:
\begin{gather}
\begin{gathered}
 \text{ If } w_k \in L^5(\mr; L^{10}(\mr^3)), \; 1\leq k \leq 5, \;\ W_2=(w_1,w_2), W_3=(w_1,w_2, w_3), \\ 
 W_5=(w_1,w_2, w_3, w_4,w_5), \\
  V_j\in L^{\frac{5}{5-j}}(I, L^{\frac{10}{5-j}}(\mr^3)), \; 1\leq j \leq 4, \;\ V_5 \in L^\infty(\mr\times \mr^3), \\ 
  \text{ then for } \ga \text{ such that } \ga=(\ga_1,\ldots,\ga_j),  |\ga|\leq j, \;\  j\leq 5,  \\
||V_j W_j^\gamma ||_{L^1(\mr;L^2(\mr^3))}\leq ||V_j||_{L^{\frac{5}{5-j}}(\mr;L^{\frac{10}{5-j}}(\mr^3))} 
\Pi_{k=1}^j||w_k||_{L^5(\mr;L^{10}(\mr^3))}^{\gamma_k}, \;\ j\leq 4 \\ 
||V_5 W^\gamma ||_{L^1(\mr;L^2(\mr^3))}\leq ||V_5||_{L^{\infty}(\mr\times \mr^3)} 
\Pi_{k=1}^5||w_k||_{L^5(\mr;L^{10}(\mr^3))}^{\gamma_k} \\
\end{gathered}\label{est-prodvw}
\end{gather}

Since $\mcn_-\zed=0,$ we have that if $T\in \mr,$ and
\begin{gather}
\begin{gathered}
\square \zed= \La(\zed)+  R, \\
\zed(T)=\vphi_T, \;\ \p_t \zed(T)=\psi_T.
\end{gathered}\label{eq-for-zed}
\end{gather}
Therefore $\mcr_-(\vphi_T, \psi_T, \La(\zed)+ R)=0,$ and hence $E_0(\vphi_T,\psi_T)\leq 
||\La(\zed)+ R||_{L^1((-\infty,T]; L^2(\mr^3))}.$ Then Strichartz estimate \eqref{STES} implies that for any $T,$
\begin{gather}
||\zed||_{L^5((-\infty,T]; L^{10}(\mr^3)} + E_0(\zed,\p_t \zed)(T) \leq 2 C_{10} || \La(\zed)+ R||_{L^1((-\infty,T]; L^2(\mr^3))}. \label{aux-est}
\end{gather}
But in view of \eqref{est-prodvw}, this implies that
\begin{gather}
\begin{gathered}
||\zed||_{L^5((-\infty,T); L^{10}(\mr^3)}+E_0(\zed,\p_t \zed)(T)  \leq 2 C_{10} ||R||_{L^1((-\infty,T]; L^2(\mr^3))}+ \\
%2C_{10} ||V_1||_{L^{\frac{5}{5-j}}((-\infty,T];L^{\frac{10}{5-j}}(\mr^3))} ||\zed||_{L^5((-\infty,T];L^{10}(\mr^3))}+ \\
2 C_{10}\left(\sum_{j=1}^4 ||V_j||_{L^{\frac{5}{5-j}}((-\infty,T];L^{\frac{10}{5-j}}(\mr^3))} ||\zed||_{L^5((-\infty,T];L^{10}(\mr^3))}^j + C ||\zed||_{L^5((-\infty,T];L^{10}(\mr^3))}^5\right).
\end{gathered}\label{Aux-eq-for-zed}
\end{gather}
We pick $\del>0,$  such that
\begin{gather}
2 C_{10} \del <\ha, \label{Ch-del}
\end{gather}
and we the pick $T_0$ such that 
\begin{gather}
||V_1||_{L^{\frac{5}{4}}((-\infty,T_0];L^{\frac{5}{2}}(\mr^3))}<\del.\label{boundv1}
\end{gather}

Then, provided $T< T_0,$
\begin{gather*}
2C_{10} ||V_1||_{L^{\frac{5}{5-j}}((-\infty,T];L^{\frac{10}{5-j}}(\mr^3))} ||\zed||_{L^5((-\infty,T];L^{10}(\mr^3))}\leq 
\ha ||\zed||_{L^5((-\infty,T];L^{10}(\mr^3))}.
\end{gather*}

So we can absorb this term into the left side of \eqref{Aux-eq-for-zed} and  conclude that for any $T<T_0,$ such that \eqref{boundv1} holds
\begin{gather}
\begin{gathered}
||\zed||_{L^5((-\infty,T); L^{10}(\mr^3)} \leq 4 C_{10} ||R||_{L^1((-\infty,T); L^2(\mr^3))}+ \\
4 C_{10}\left(\sum_{j=2}^4 ||V_j||_{L^{\frac{5}{5-j}}((-\infty,T];L^{\frac{10}{5-j}}(\mr^3))} ||\zed||_{L^5((-\infty,T];L^{10}(\mr^3))}^j + C ||\zed||_{L^5((-\infty,T];L^{10}(\mr^3))}^5\right).
\end{gathered}\label{boundzed1}
\end{gather}

Equation \eqref{mcet0} for $T<T_0$ follows from the following  observation, which resembles Lemma 2.2 of \cite{BahGer}:
\begin{lemma}\label{poly-ineq} Let $M(T)\geq 0$ be a continuous function on and interval $I=(-\infty,\beta],$ $I=[\alpha,\beta]$  or $I=[\alpha,\infty),$ and suppose that  either that $\lim_{T\rightarrow -\infty} M(T)=0$ or that $M(\alpha)=0.$   If 
\begin{gather}
M(T)\leq a + \sum_{j=2}^5 C_j M(T)^j, \;\  a>0, \; C_j>0,
\text{ are such that } \sum_{j=2}^5 2^j C_j a^{j-1}<1, \label{eq-M}
\end{gather}
then  $M(T) \leq 2 a.$
\end{lemma}
\begin{proof} Set $x=M(T)$ and consider the polynomial $p(x)= x-(a+ C_2 x^2+ C_3 x^3+ C_4 x^4+ C_5 x^5).$  Obviously $p(0)=-a<0$ and the second condition in \eqref{eq-M} guarantees that $p(2a)>0.$  Since $M(T)$ is continuous and $M(\alpha)=0$ and the roots of $p(x)$ are discrete, $M(T)\leq x^*,$ where $x^*$ is the smallest value of $x>0$ for which the graph of $p(x)$ crosses the real axis. Then $M(T)\leq x^*< 2a.$  The same argument applies to the case $\lim_{T\rightarrow -\infty} M(T)=0.$
\end{proof}
We apply    \eqref{eq-M} to \eqref{boundzed1} with
\begin{gather*}
M(T)=||\zed||_{L^5((-\infty,T); L^{10}(\mr^3)}, \;\ a=4 C_{10} ||R||_{L^1(\mr; L^2(\mr^3))}, \text{ and } \\
C_j=4C_{10}||V_j||_{L^{\frac{5}{5-j}}(\mr;L^{\frac{10}{5-j}}(\mr^3))}, \;\ 1\leq j \leq 4, \;\ C_5=4C_{10}C.
\end{gather*}
But, because of \eqref{estR} $a\leq C|\veps|^5,$ with $C$ independent of $T,$ and so we can pick $\eps_0>0$ such that $\sum_{j=2}^5 2^j C_j a^{j-1}<1$ for $\eps<\eps_0,$ and for this choice of $\eps_0,$ 
\begin{gather}
||\zed||_{L^5((-\infty,T]; L^{10}(\mr^3)} \leq C |\veps|^5, \text{ provided }T\leq T_0, \text{ and } |\veps|<\eps_0. \label{sest0}
\end{gather}

We substitute estimate \eqref{sest0} in \eqref{Aux-eq-for-zed} and we conclude that
\begin{gather}
E_0(Z,\p_t Z)(T) \leq K_0 |\veps|^5,  \text{ provided }T\leq T_0, \text{ and } |\veps|<\eps_0. \label{EST-E0}
\end{gather}

Given  $T_1>T_0,$   as in the proof of Proposition \ref{sol123}, we partition the interval $[T_0,T_1]$ into $N$ intervals 
$I_j=[a_j,b_j],$ $a_0=T_0$ and $b_N=T_1$ such that

\begin{gather}
||V_1||_{L^{\frac{5}{4}}(I_j;L^{\frac{5}{2}}(\mr^3))}<\del, \text{ with } \del  \text{ as in \eqref{Ch-del}}.\label{boundvj}
\end{gather}

Strichartz estimate \eqref{STES} gives that for $t\in [a_0,a_1],$ and $K_0$ as in \eqref{EST-E0},
\begin{gather}
\begin{gathered}
||\zed||_{L^5([a_0,t]; L^{10}(\mr^3))} +E_0(\zed, \p_t \zed)(t)  \leq  2C_{10} ||R||_{L^1([a_0,a_1]; L^2(\mr^3))}+ 2C_{10} K_0\eps^5+ \\
 2 C_{10}\left(\sum_{j=1}^4 ||V_j||_{L^{\frac{5}{5-j}}([a_0,t];L^{\frac{10}{5-j}}(\mr^3))} ||\zed||_{L^5(([a_0,t];L^{10}(\mr^3))}^j + C ||\zed||_{L^5([a_0,T];L^{10}(\mr^3))}^5\right).
\end{gathered}\label{boundzed11}
\end{gather}
As above, we use \eqref{boundvj} and \eqref{Ch-del} to absorb the term $2C_{10} ||V_1||_{L^{\frac{5}{4}}([a_0,t];L^{\frac{5}{2}}(\mr^3))} ||\zed||_{L^5(([a_0,t];L^{10}(\mr^3))}$ onto the left hand side,
and we conclude that
\begin{gather}
\begin{gathered}
|\zed||_{L^5([a_0,t]; L^{10}(\mr^3))} +E_0(\zed, \p_t \zed)(t)  \leq  2C_{10} ||R||_{L^1([a_0,a_1]; L^2(\mr^3))}+ 2C_{10} K_0\eps^5+ \\
 2 C_{10}\left(\sum_{j=2}^4 ||V_j||_{L^{\frac{5}{5-j}}([a_0,t];L^{\frac{10}{5-j}}(\mr^3))} ||\zed||_{L^5(([a_0,t];L^{10}(\mr^3))}^j + C ||\zed||_{L^5([a_0,T];L^{10}(\mr^3))}^5\right).
\end{gathered}\label{boundzed11-Copy}
\end{gather}
Another application of \eqref{estR} and Lemma \ref{poly-ineq} shows that there exists $\eps_0$ and $C>0$ such that for $\eps_j<\eps_0,$ 
\begin{gather*}
||\zed||_{L^5([a_0,t]; L^{10}(\mr^3))}\leq  C |\veps|^5 ,  \;\ t\in [a_0,a_1],
\end{gather*}
and substituting this in \eqref{boundzed11} gives
\begin{gather*}
E_0(\zed, \p_t \zed)(t) \leq K_1 |\veps|^5 , \text{ provided }  t\leq a_1, \; |\veps|<\eps_0
\end{gather*}
 We repeat the same argument for the interval $[a_1,a_2]$ and subsequent intervals.  The key point here is that with $\delta$ such that $2C_{10}\del<\ha,$ one can absorb the term 
 \begin{gather*}
2C_{10} ||V_1||_{L^{\frac{5}{4}}([a_j,t];L^{\frac{5}{2}}(\mr^3))} ||\zed||_{L^{5}([a_j,t];L^{10}(\mr^3))}
\end{gather*}
onto the left hand side and arrive at an estimate just like \eqref{boundzed11} for the interval $[a_j,a_{j+1}].$ The choice of $\del$ remains fixed.  We conclude that there exists $\eps_0=\eps_0(N),$  and $C=C(N)>0$ such that
\begin{gather*}
 ||\zed||_{L^5([a_j,t]; L^{10}(\mr^3)}+E_0(\zed, \p_t \zed)(t) \leq C |\veps|^5, \text{ provided }  a_j\leq t \leq a_{j+1} \text{ and } |\veps|<\eps_0. \label{sest01}
 \end{gather*}
Grouping these terms together we find that there exists $\eps_0=\eps_0(N)>0$ and $C=C(N)$ such that
 \begin{gather}
 \begin{gathered}
  ||\zed||_{L^5([T_0,T]; L^{10}(\mr^3)} \leq C |\veps|^5 \text{ and } 
  E_0(\zed, \p_t \zed)(T) \leq K |\veps|^5, \\
\text{ provided }  T_0 \leq t \leq T_1 \text{ and } |\veps|<\eps_0.
 \end{gathered}\label{sest011}
 \end{gather}

This process seems to fail, as the constants depend on $N,$ and we cannot let $T_1\nearrow \infty,$ but the point is that we can pick $T_1$ such that
\begin{gather}
||V_1||_{L^{\frac{5}{4}}([T_1,\infty)];L^{\frac{5}{2}}(\mr^3))}<\del.\label{boundv2}
\end{gather}

Then Strichartz estimate \eqref{STES} shows that, for $K$ as in \eqref{sest011},
\begin{gather}
\begin{gathered}
||\zed||_{L^5([T_1,T]; L^{10}(\mr^3)} +E_0(\zed, \p_t \zed)(T) \leq  2C_{10} ( ||R||_{L^1([T_1,T]; L^2(\mr^3))}+K |\veps|^5)+  \\\
2 C_{10}\left(\sum_{j=1}^4 ||V_j||_{L^{\frac{5}{5-j}}([T_1,T];L^{\frac{10}{5-j}}(\mr^3))} ||\zed||_{L^5([T_1,T];L^{10}(\mr^3))}^j + C ||\zed||_{L^5([T_1,T];L^{10}(\mr^3))}^5\right).
\end{gathered}\label{boundzed1-N}
\end{gather}
Because $2C_{10}\del<\ha$ and \eqref{boundv2} holds,  we can absorb the term 
$||V_1||_{L^{\frac{5}{4}}([T_1,T];L^{\frac{5}{2}}(\mr^3))} ||\zed||_{L^5([T_1,T];L^{10}(\mr^3))}$ onto the left hand side and conclude that
\begin{gather}
\begin{gathered}
||\zed||_{L^5([T_1,T]; L^{10}(\mr^3))} \leq  2 C_{10} ( ||R||_{L^1([T_1,T]; L^2(\mr^3))}+K |\veps|^5)+  \\
4 C_{10}\left(\sum_{j=2}^4 ||V_j||_{L^{\frac{5}{5-j}}([T_1,T];L^{\frac{10}{5-j}}(\mr^3))} ||\zed||_{L^5([T_1,T];L^{10}(\mr^3))}^j + C ||\zed||_{L^5([T_1,T];L^{10}(\mr^3))}^5\right).
\end{gathered}\label{boundzed2}
\end{gather}

Again we apply \eqref{estR} and Lemma \ref{poly-ineq} to conlcude that that  there exist $C>0$ and $\eps_0>0$ such that for $\eps<\eps_0$ 
\begin{gather}
||\zed||_{L^5([T_1,T]; L^{10}(\mr^3)} \leq C |\veps|^5, \text{ provided }  T_1 \leq  T  \text{ and } |\veps|<\eps_0.\label{SEST-B1}
\end{gather}
Substituting this into \eqref{boundzed1-N} and using \eqref{estR} we conclude that
\begin{gather}
  E_0(\zed, \p_t \zed)(T) \leq K |\veps|^5, \text{ provided }  T_1 \leq  T  \text{ and } |\veps|<\eps_0. \label{SEST-B2}
  \end{gather}

Combining \eqref{sest0}, \eqref{sest011},  \eqref{SEST-B1} and \eqref{SEST-B2} we deduce that there exists $\eps_0>0$ such that
 \begin{gather}
 \begin{gathered}
  ||\zed||_{L^5((-\infty,T]; L^{10}(\mr^3)} \leq C |\veps|^5 \text{ and } 
  E_0(\zed, \p_t \zed)(T) \leq K |\veps|^5, \text{ for all } T \in \mr, \text{ provided }  |\veps|<\eps_0.
 \end{gathered}\label{SEST-B3}
 \end{gather}

This proves the first estimate in \eqref{mcet0}, so it remains to prove the second estimate in \eqref{mcet0}.   But once we have \eqref{SEST-B3} we use \eqref{est-prodvw}  to show that
 \begin{gather*}
 ||\square \zed||_{L^1((-\infty,T]; L^{2}(\mr^3)} \leq ||\La(\zed)||_{L^1((-\infty,T]; L^{2}(\mr^3)} +  ||R||_{L^1((-\infty,T]; L^{2}(\mr^3)} \leq \\
 \sum_{j=1}^4 ||V_j||_{L^{\frac{5}{5-j}}((-\infty,T];L^{\frac{10}{5-j}}(\mr^3))} ||\zed||_{L^5((-\infty,T];L^{10}(\mr^3))}^j + \\ C ||\zed||_{L^5((-\infty,T];L^{10}(\mr^3))}^5 +  ||R||_{L^1((-\infty,T]; L^{2}(\mr^3)}\leq C |\veps|^5.
 \end{gather*}
 This proves the second estimate in \eqref{mcet0} and it concludes the proof of Proposition \ref{solw4}.
\end{proof}

\section{Radiation Fields of Conormal Spherical Waves}

  As discussed in Section \ref{outline} pick $\Upsilon_0\in C_0^\infty(\mr\times \ms^2)$  to be radial, or independent of $\omega,$ and we pick $\Upsilon_j,$ $j=1,2,3,4$  of the form
\begin{gather}
\begin{gathered}
\Ups_j(s,\omega)= \p_s\left[ (s-s_j- \lan \omega,z_j\ran)_+^m \chi(s-s_j- \lan \omega,z_j\ran)\right],  \; j=1,2,3,4, \;\ m>0, \\
\chi \in C_0^\infty(\mr) \text{ supported in } [-1,1],
\end{gathered}\label{def-upsj}
\end{gather}
which correspond to the backward radiation field of four spherical waves.    Indeed, notice that 
\begin{gather*}
%\zeta_j^+(t,z) =  |z-z_j|^{-1} (t+s_j- |z-z_j|)_+^{m} \chi_j(t-s_j-|z-z_j|),\\
\zeta_j^-(t,z) =  |z-z_j|^{-1} (t-s_j+ |z-z_j|)_+^{m} \chi(t-s_j+|z-z_j|), 
\end{gather*}
for $j=1,2,3,4,$ satisfy $\square \zeta^-_j=0.$ Also notice that if $z=r \omega,$
\begin{gather*}
|z-z_j|= \left(r^2-2r\lan \omega, z_j \ran + |z_j|^2\right)^\ha= r + O(1), \\
t-s_j+|z-z_j|= t-s_j+ \left( r^2+|z_j|^2-2r \lan \omega, z_j\ran\right)^\ha= t-s_j+ r\left(1-\frac{2}{r}\lan \omega, z_j\ran +\frac{|z_j|^2}{r^2}\right)^\ha=\\
t-s_j+ r-\lan \omega,z_j\ran +O(r^{-1}),
\end{gather*}
and therefore, for $j=1,2,3,$
\begin{gather}
\begin{gathered}
%\mcn_+ \zeta_j^+(s,\omega)=  \lim_{r\rightarrow \infty} r \p_t \zeta_j(s+r,r\omega)=
%\p_s\left((s+s_j- \lan \omega,z_j\ran)_+^{m} \chi_j(s-s_j+ \lan \omega,z_j\ran)\right)= \Ups_j(s,\omega), \\
\mcn_- \zeta_j^-(s,\omega)=  \lim_{r\rightarrow \infty} r \p_s \zeta_j(s-r,r\omega)=
\p_s\left((s-s_j- \lan \omega,z_j\ran)_+^{m} \chi(s-s_j- \lan \omega,z_j\ran)\right)= \Ups_j(s,\omega).
\end{gathered}\label{radfpm}
\end{gather}

But we need to  analyze the singularities of the solutions $w_j(y)$ to \eqref{defwj} with backward radiation field given by $\Ups_j.$  Recall that from the discussion in Section \ref{outline}, with a suitable choice of $\Ups_0,$  $u_0(t,x)\in C^\infty(\mr\times \mr^3),$ where $u_0(t,x)$ is the solution of \eqref{rad-u0}.

 Given $u_0,$ we want to show that the solution to \eqref{defwj},  has a singularity expansion of the form
\begin{gather}
\begin{gathered}
w_j^-(t,z)\sim \sum_{k=0}^\infty \beta_k^-(z)\chi(t-s_j+|z-z_j|) |z-z_j|^{-1} \left(t-s_j+ |z-z_j|\right)_-^{m+k},\;
  \;\ \beta_0^-(z)=1, \text{ for } t< s_j, \\
  w_j^+(t,z) \sim \sum_{k=0}^\infty \beta_k^+(z)\chi(t+s_j-|z-z_j|) |z-z_j|^{-1} \left(t-s_j+ |z-z_j|\right)_+^{m+k},\;
  \;\ \beta_0^+(z)=1, \text{ for } t> s_j.
\end{gathered} \label{sing-expw}
\end{gather}
In terms of singularities, this corresponds to a spherical wave coming from $t=-\infty,$ which collapses to a point at 
$\{t=s_j\}$ and re-emerges as a spherical wave for $t>0.$  Friedlander \cite{Fried2} named this a hourglass wave.

  Since we are only concerned with the singularities on the spherical wave, we treat the cut off $\chi$ as being equal to one.

 Away from $z=z_j,$ we can write
\begin{gather}
\begin{gathered}
f'(u_0)\sim \sum_{k=0}^\infty \gamma_k^-(z) \left(t-s_j+ |z-z_j|\right)^{k}, \;\ t< s_j, \\
f'(u_0)\sim \sum_{k=0}^\infty \gamma_k^+(z)  \left(t-s_j- |z-z_j|\right)^{k}, \;\ t> s_j.
\end{gathered}\label{exp-fu0} 
\end{gather}
We substitute this expansion into the wave equation \eqref{defwj} and also  use the facts that 
\begin{gather*}
\square\left( |z-z_j|^{-1}\left(t- s_j+ |z-z_j|\right)_\pm^{k}\right)=0, \text{ for any } k \in \mn, 
\end{gather*}
 and that $\beta_0^-=1.$  To simplify the notation, we  denote
 \begin{gather*}
 \vtheta_{j-}^{k}= |z-z_j|^{-1}\left(t-s_j+ |z-z_j|\right)_-^{k}.
 \end{gather*} 
 
\begin{gather*}
(\square + f'(u_0)) w_j^-\sim \square \sum_{k=0}^\infty \beta_k^-\;  \vtheta_{j-}^{m+k}+ 
 \sum_{l+k=0}^\infty \gamma_l^- \beta_k^- \vtheta_{j-}^{m+k+l} = \\
- \sum_{k=1}^\infty \left(\Delta\beta_k^--  2\frac{z-z_j}{|z-z_j|^2}\cdot \nabla_z \beta_k^- \right) \vtheta_{j-}^{m+ k}
  \;\ - \\ \sum_{k=1}^\infty 2(m+k) \left(\frac{z-z_j}{|z-z_j|}\cdot \nabla_z\beta_k^- \right)\vtheta_{j-}^{m+k-1} \; + 
  \sum_{l+k=0}^\infty \gamma_l \beta_k^-\; \vtheta_{j-}^{m+k+l}, 
  \end{gather*}
 Next we  introduce  polar coordinates with vertex at $z_j,$ and use $r_j$ to denote the distance to $z_j,$  and $\theta= \frac{z-z_j}{|z-z_j|},$ so $\frac{z-z_j}{|z-z_j|}\cdot \nabla_z=\p_{r_j},$  and we obtain
  \begin{gather*}
  (\square + f'(u_0)) w_j^-\sim   (\ga_0^--2(m+1) \p_{r_j} \beta_1^-)  \vtheta_{j-}^{m}+ \\
  \sum_{k=1}^\infty \left[ 2(m+k+1) \p_{r_j} \beta_{k+1}^- +( \Delta-\frac{2}{r_j} \p_{r_j})\beta_{k}^- +\sum_{l+\mu=k} \ga_l^- \beta_{\mu}^-\right]\vtheta_{j-}^{m+k}.
  \end{gather*}

%Similarly, we obtain 
%\begin{gather*} 
%(\square + f'(u_0)) w_j^+\sim  (\ga_0^+-2(m+1) \p_{r_j} \beta_1^+) \left(t+s_j- |z-z_j|\right)_+^{m}- \\
 % \sum_{k=1}^\infty \left[ 2(m+k+1)\p_{r_j} \beta_{k+1}^+ +\Delta\beta_{k}^+ +\sum_{l+\mu=k} \ga_l^+ \beta_{\mu}^+\right]\left(t+s_j- |z-z_j|\right)_+^{m+k}.
 % \end{gather*}
  This gives a series of transport equations for the coefficients $\beta_k^-:$
  \begin{gather}
  \begin{gathered}
  \ga_0^--2(m+1) \p_{r_j} \beta_1^-=0, \;\ \lim_{r_j\rightarrow \infty} \beta_1(r_j,\theta)=0, \\
   2(m+k+1)\p_{r_j} \beta_{k+1}^- +(\Delta-\frac{2}{r_j} \p_{r_j})\beta_{k}^- +\sum_{l+\mu=k} \ga_l^- \beta_{\mu}^-=0, \;\ \lim_{r\rightarrow \infty} \beta_{k+1}(r_j,\theta)=0, \;\ k\geq 1.
\end{gathered}\label{transp-eq-}
      \end{gather}
      
For $t>s_j$ one has the expansion
\begin{gather}
\begin{gathered}
w_j^+(t,z)\sim \sum_{k=0}^\infty \beta_k^+(r\omega)\chi(t-s_j-|z-z_j|) |z-z_j|^{-1} \left(t-s_j- |z-z_j|\right)_-^{m+k}, \;\ \beta_0^+=1.
\end{gathered} \label{sing-expw+}
\end{gather}
The coefficients of the expansion $\beta_k^+(z)$ satisfy a similar set of transport equations:
\begin{gather}
  \begin{gathered}
 \beta_0^+ \ga_0^+-2(m+1) \p_{r_j} \beta_1^+=0, \;\ \beta_1^+(z_j)=\beta_1^-(z_j),\\
   2(m+k+1)\p_{r_j} \beta_{k+1}^+ +(\Delta-\frac{2}{r_j} \p_{r_j})\beta_{k}^+ +\sum_{l+\mu=k} \ga_l^+ \beta_{\mu}^+=0, \;\
   \beta_{k+1}^+(z_j)=\beta_{k+1}^-(z_j),\\
\end{gathered}\label{transp-eq+}
\end{gather}
Here $\ga_k^+$ correspond to the analogue coefficients of the expansion \eqref{exp-fu0}.
The scattering operator for the operator $\square+f'(u_0),$ $\mca$ satisfies, for any   $N\in \mn,$
\begin{gather}
\begin{gathered}
\mca(\Ups_j)(s,\omega)- \lim_{r\rightarrow \infty} \sum_{k=0}^N \beta_{k}^+(z) \p_s\left((s-s_j+\lan z_j,\omega\ran)_+^{m+k} \; \chi(s-s_j+\lan z_j,\omega\ran\right)\in C^N(\mr\times \ms^2), 
\end{gathered}\label{exp-smatrix}
\end{gather}

\section{Propagation of Singularities for Radiation Fields}

 We will  need a general result about propagation of singularities for radiation fields and we recall a result which is essentially contained in the the work of S\'a Barreto and Wunsch \cite{SaBWun}:
\begin{prop}\label{new-sing-111}   Let $\La\subset T^*\mr^4\setminus 0$ be a $\CI$ conic Lagrangian submanifold which is contained in 
$\{p=\sigma_2(\square)=0\}.$
Let $(t,z),$ $z=(x_1,x_2,x_3)\in \mr^3,$ be coordinates in $\mr^4.$ Let $r=|z|$ and $\omega \in \ms^2,$ $\omega=\frac{z}{|z|}.$ Let $x=\frac{1}{r},$ $s=t-\frac{1}{x}$ where $r=|z|,$ $z=r\omega.$  Let $\widetilde{\La}$ denote the image of $\La$ under this coordinate change.  Then $\widetilde{\La}$ intersects  $\{x=0\}$ transversally and  
$\widetilde{\La}$ extends as a $C^\infty$ conic Lagrangian submanifold of 
$T^*U\setminus 0,$ where $U=(-\eps,\eps)_x\times \mr_s\times \ms^2_{\omega}.$ Moreover,
$\widetilde{\La}\cap \{x=0\}= \La_\infty$ is a $C^\infty$ conic Lagrangian submanifold of $T^*(\mr_s\times \ms_\omega^2)\setminus 0.$

If locally in $T^*U\setminus 0$ and near $\{x=0\},$ $\widetilde \La=N^*\widetilde \Sigma,$ where $\widetilde\Sigma\subset U,$ is a $C^\infty$ hypersurface, 
then $\widetilde\Sigma$ intersects $\{x=0\}$ transversally at $\Sigma_{\infty}=\widetilde\Sigma\cap \{x=0\},$ and $\La_\infty=N^*\Sigma_\infty.$ We shall call
$\Sigma_\infty$ and $\La_\infty$ respectively  the forward radiation patterns of $\Sigma$ and $\La.$

If $u \in I^{m}(U,\widetilde\La)$ is a Lagrangian distribution, then 
$u|_{x=0}\in I^{m+\oq}(\mr\times \ms^2, \widetilde{\La}_\infty)$  and its symbol is equal to 
$\sigma(u)|_{\widetilde{\La}_\infty}.$ If $u$ is elliptic, so is $u|_{\{x=0\}}.$
\end{prop}
\begin{proof} Written in the variables $(s,r,\omega),$ $s=t-r,$  as $r\rightarrow \infty,$ the operator $P= r \square r^{-1}$ becomes
\begin{gather*}
P= \p_{r}(2\p_s-\p_r) -\frac{1}{r^2} \Delta_\omega,
\end{gather*}
where $\Delta_\omega$ is the Laplacian on $\ms^2.$ If one sets $x=r^{-1},$ then 
\begin{gather*}
-x^{-2} P= \mcp=\p_x( 2\p_s+ x^2 \p_x) + \Delta_\omega.
\end{gather*}
This is a compactification of $\mr^n\setminus 0,$ which has a natural extension to a neighborhood of $\{x=0\}.$

The Hamilton vector field of $\varrho=\sigma_2(\mcp)=2\mu\xi+x^2\xi^2+ h(\omega,\varkappa),$ where $\xi$ is the dual variable to $x,$ $\mu$ is the dual to $s,$  and $\varkappa$ is the dual to $\omega,$ is given by
\begin{gather*}
H_{\varrho}= 2(\mu+x^2\xi) \p_x+ 2\xi \p_s +H_{h}.
\end{gather*}
But on the characteristic variety of $\square,$ and away from the zero section of $T^* \mr^4,$ $\tau^2\not=0.$ But $\tau=\mu,$ and so we conclude that near $x=0,$ and on $\{\varrho=0\},$ $H_{\varrho}$ is transversal to $\{x=0\}.$
So, as observed by S\'a Barreto and Wunsch \cite{SaBWun},  since $\widetilde\La\subset \{\varrho=0\},$  $H_{\varrho}$ is tangent to $\widetilde\La,$ and so  the Lagrangian submanifold $\widetilde\La$ extends across $\{x=0\}$ as the union of integral curves of $H_\varrho$ which start on $\widetilde{\La}$ in $x<0.$ $\widetilde{\La}$ is a   $C^\infty$ Lagrangian submanifold of $T^*U.$ 

Suppose that $\widetilde\La=N^* \widetilde\Sigma,$ then the projection $\Pi: \widetilde\La\longmapsto U$ is a diffeomorphism. But then, since $H_\varrho$ is transversal to $\{x=0\}$ and tangent to $\widetilde\La,$ for a point $p\in \widetilde\La\cap \{x=0\},$ the tangent space 
$T_{p} \widetilde \La=T_p \La_\infty\oplus H_\varrho(p).$ Therefore, the projection $\Pi_\infty: \La_\infty \longmapsto \{x=0\}$ is a diffeomorphism.

Suppose that  $u \in I^{m}(\mr^4,\widetilde\La),$ is a Lagrangian distribution, then near $p\in \{x=0\},$ $u$ is given by an oscillatory integral
\begin{gather*}
u(x,s,\omega)= \int_{\mr^N} e^{i\phi(x,s,\omega,\theta)} a(x,s,\omega,\theta) d\theta, \; 
a\in S^{m+1-\frac{N}{2}} (U \times \mr^N_{\theta}),
\end{gather*}
where $\phi(x,s,\omega,\theta)$ parametrizes $\widetilde\La$ in the sense that

\begin{gather*}
\p_\theta \phi=0, \;\ \p_x\phi=\xi, \;\ \p_s \phi=\mu, \;\ \p_\omega \phi=\varkappa.
\end{gather*}

Then of course, $\phi(0,s,\omega,\theta)$ parametrizes $\La_\infty$ and  
\begin{gather*}
u(0,s,\omega)= \int_{\mr^N} e^{i\phi(0,s,\omega,\theta)} a(0,s,\omega,\theta) d\theta, \; 
a\in S^{m+1-\frac{N}{2}} (U \times \mr^N_{\theta}).
\end{gather*}
Since the dimension drops by one, the order of the Lagrangian distribution goes up by $\oq.$  If $u$ is elliptic, so is $u|_{\{x=0\}}.$
\end{proof}

\section{Singularities Produced by the Interaction of Three Waves}

In this section we consider the case where $\Ups_4=0$ and $\Ups_j,$ $j=1,2,3$ are given by \eqref{def-upsj}.  In particular we are interested in analyzing the singularities of $w_{1,1,1,0}.$   Since $\Ups_4=0,$ it follows that $w_4=0$

According to the discussion above, the solutions $w_\alpha,$ with $|\alpha|=1$ to \eqref{defwj}, are conormal to the half  light cones
\begin{gather*}
\Sigma_{j}^-=\{t-s_j+ |z-z_j|=0\}, \;\ j=1,2,3, \text{ if } t< s_j, \\
\Sigma_j^+=\{t-s_j-|z-z_j|=0\}, \;\ j=1,2,3,  \text{ if } t>s_j,
\end{gather*}
Here, according to \eqref{convw},  we identify $w_{1,0,0,0}= w_1,$ $w_{0,1,0,0}= w_2,$ $w_{0,0,1,0}= w_3.$ 
 
We know the functions $w_\alpha,$ $|\alpha|=2,$ satisfy
\begin{gather*}
\left(\square + f'(u_0)\right) w_{\alpha} = f''(u_0)\sum_{|\beta_1|=|\beta_2|=1} w_{\beta_1} w_{\beta_2}, \\
\mcn_-(w_{\alpha})=0,
\end{gather*}

 We know from a result of Bony \cite{Bony3} that for $t<s_j$
$w_{1,1,0,0}$ is conormal to $\Sigma_1^-\cup \Sigma_2^-,$ and for $t>s_j$
$w_{1,1,0,0}$ is conormal to $\Sigma_1^+\cup \Sigma_2^+,$ and similarly
$w_{1,0,1,0}$ is conormal to $\Sigma_1^{\mp}\cup \Sigma_3^{\mp}$ and $w_{0,1,1,0}$ is conormal to $\Sigma_2^{\mp}\cup \Sigma_3^{\mp},$  
depending of whether if $t<s_j$ or $t>s_j.$
Hence,  based on Proposition \ref{new-sing-111},  the singularities of $\Xi_\alpha=\mcn_+(w_\alpha),$ 
$|\alpha|=2$ are contained in $\Sigma_{j\infty}^+,$ $j=1,2,3,$ where
\begin{gather}
\Sigma_{j,\infty}^+=\{s-s_j+\lan \omega, z_j\ran =0\}, \; j=1,2,3, \label{def-Sji}
\end{gather}

Our goal is to hunt for singularities of $\Xi_{1,1,1,0}=\mcn_+(w_{1,1,1,0})$ which are not cointained in $\Sigma_{j\infty}^+,$ $j=1,2,3.$ These singularities will be generated by the nonlinearity, and will give information about $f^{(3)}(u_0).$

According to \eqref{defw123},   
\begin{gather}
\begin{gathered}
(\square +f'(u_0) )w_{1,1,1,0}=- f''(u_0)( w_{1,0,0,0} w_{0,1,1,0}+  w_{0,1,0,0} w_{1,0,1,0}+ w_{0,0,1,0} w_{1,1,0,0})-\\
\frac{1}{3!} f^{(3)}(u_0) w_{1,0,0,0}w_{0,1,0,0}w_{0,0,1,0}, \\
 \mcn_- w_{1,1,1}=0.
\end{gathered}\label{eq-w111}
\end{gather}

The distributions $w_\beta,$ $|\beta|=1$ are conormal to $\Sigma_{\beta}^\pm,$ according to the identification \eqref{convw}  also used above.  The distributions
$w_{\beta_1+\beta_2},$ $|\beta_1|=|\beta_2|=1,$ are conormal to $\Sigma_{\beta_1}^\pm\cup \Sigma_{\beta_2}^\pm.$ The products
$w_{1,0,0,0} w_{0,1,1,0},$ $w_{0,1,0,0} w_{1,0,1,0}$ and $w_{0,0,1,0} w_{1,1,0,0}$ are conormal to $\Sigma_1\cup \Sigma_2\cup \Sigma_3.$ The same is true for the triple product  $w_{1,0,0,0}w_{0,1,0,0}w_{0,0,1,0}.$ According to a theorem of Melrose and Ritter \cite{MelRit} and Bony \cite{Bony6,Bony1}, $w_{1,1,1,0}$ is conormal to
$\Sigma_1\cup \Sigma_2\cup \Sigma_3\cup \mcq^-\cup \mcq^+,$ where $\mcq_\pm$ are the hypersurfaces emanating from $\Gamma_-=\Sigma_{1}^-\cap \Sigma_{2}^-\cap \Sigma_{3}^-$ and $\Gamma_+=\Sigma_{1}^+\cap \Sigma_{2}^+\cap \Sigma_{3}^+$ respectively.  Here
\begin{gather}
\La_{\mcq^\pm}= \bigcup_{\mu>0} \exp \left( \mu H_p \left( (N^*\Gamma^\pm\setminus 0) \cap \{p=0\}\right) \right), \text{ and } \mcq^\pm=\Pi(\La_{\mcq^\pm}). \label{def-Qpm}
\end{gather}

According to Proposition \ref{new-sing-111}, we define
\begin{gather}
\La_{\mcq^\pm,\infty}= \widetilde{\La_{\mcq^\pm}} \cap \{x=0\}.\label{Lag-Q}
\end{gather}
We will show that the singularities of $\Xi_{1,1,1,0}=\mcn_+(w_{1,1,1,0})$ are contained in 
\begin{gather*}
\left(\bigcup_{j=1}^3N^*\Sigma_{j,\infty}\setminus 0\right) \cup \left(\La_{\mcq_-,\infty} \cup \La_{\mcq_+,\infty} \right).
\end{gather*}
  We will compute the principal symbol of $w_{1,1,1,0}$ on $\La_{\mcq^-,\infty}$ away from $\cup_{j=1}^3N^*\Sigma_{j,\infty}.$
  
   The key point is to show that even though the radiation fields of the terms 
 \begin{gather*}
 E_+\left( f''(u_0)(w_{1,0,0,0} w_{0,1,1,0}+  w_{0,1,0,0} w_{1,0,1,0}+ w_{0,0,1,0} w_{1,1,0,0})\right) \\  \text{ and } 
E_+\left( f^{(3)}(u_0)w_{1,0,0,0} w_{0,1,0,0} w_{0,0,1,0})\right),
\end{gather*}
 will be singular at $\bigcup_{\pm} \La_{\mcq^\pm,\infty}\setminus 0,$ the singularities of the latter term are stronger.

  We already know that for $\Ups_j,$ $j=1,2,3,$ given by \eqref{def-upsj}, the solution $w_j$ to \eqref{defwj} has an asymptotic expansion  given by \eqref{sing-expw}.  So we find that for $t<s_j,$ $w_j$ is a conormal distribution to $\Sigma_j^-=\{t-s_j+|z-z_j|=0\},$ and  if we denote
\begin{gather*}
y_j= t-s_j+|z-z_j|, \; j=1,2,3, \text{ away from } z=z_j, \text{ and so } \Sigma_j=\{y_j=0\}.
\end{gather*}
Since $w_{1,0,0,0}= Z_1(y)y_{1+}^m+$ smoother terms, $w_{0,1,0,0}= Z_2(y)y_{2+}^m+$ smoother terms, $w_{0,0,3,0}= Z_3(y)y_{3+}^m+$ smoother terms,  where $Z_j(y)= |z-z_j|^{-1}$ written with respect to $y.$
It follows that 
\begin{gather}
w_{1,0,0,0}  \in I^{-m-\tha}(\mr^4,\Sigma_1^-),\;\  w_{0,1,0,0} \in I^{-m-\tha}(\mr^4,\Sigma_2^-),
\text{ and  } w_{0,0,1,0} \in I^{-m-\tha}(\mr^4,\Sigma_3^-).  \label{lag-reg-wj}
\end{gather}

One way to  analyze the singularities of  $w_{1,1,1,0}$ is to use the calculus of paired Lagrangian distributions established by Greenleaf and Uhlmann  \cite{GreUhl}.  This is the approach used in the work  of Kurylev, Lassas and Uhlmann \cite{KLU} and Lassas, Uhlmann and Wang \cite{LUW}.   In fact we can just quote Proposition 3.7 of \cite{LUW}.  
\begin{prop}\label{reg-w111-N} (Lassas, Uhlmann and Wang \cite{LUW}) Let $\Gamma_\pm=\Sigma_1^\pm\cap \Sigma_2^\pm \cap \Sigma_3^\pm,$  and let 
$\La_{\mcq_\pm}$ be defined in \eqref{def-Qpm}. If $f^{(3)}(u_0)$ does not vanish on a segment $\mci_\pm \subset \Gamma_\pm,$ then away from $N^* \Sigma_j^\pm$ and $N^*(\Gamma_\pm),$ $w_{1,1,1,0}\in I^{-3m-4}(\mr^4, \La_{\mcq^\pm})$ on the portion of $\La_{\mcq^\pm}$ emanating from $\mci^\pm.$  In the case of $\Gamma_-,$ according to \eqref{lag-reg-wj}, the principal part of $w_{1,1,1,0}$ on the subset of $\mcq^-$ emanating from $\mci_-$ is given by $\frac{1}{3!}E_+\left( (\mcz f^{(3)}(u_0)|_{\mci_\pm}) y_{1+}^my_{2+}^my_{3+}^m\right),$ where $\mcz=Z_1Z_2Z_3.$  
\end{prop}
Using this result, one can compute the singularities of $\mcn_+(w_{1,1,1,0}):$
\begin{prop}\label{add-sing-w} Let $\Ups_j,$ $j=1,2,3$ be given by \eqref{def-upsj}. Let $\Sigma_{j,\infty},$  given by \eqref{def-Sji}, 
denote the forward radiation patterns of $\Sigma_j^+.$ Let $\Gamma_\pm=\Sigma_1^\pm\cap \Sigma_2^\pm \cap \Sigma_3^\pm,$  and let 
$\La_{\mcq_\infty^\pm}$ be defined in \eqref{Lag-Q}. Let $w_{1,1,1,0}$ be the corresponding solution to \eqref{eq-w111}.  Then, for suitably chosen $z_j,$ and $s_j,$ $j=1,2,3,$ and away from $\Sigma_{j,\infty}^{\pm},$  
 
 \begin{gather*}
 \Xi_{1,1,1,0}=\mcn_+(w_{1,1,1,0})\in I^{-3m-\frac{15}4}(\mr\times \ms^2; \La_{\mcq^-,\infty})+ I^{-3m-\frac{15}4}(\mr\times \ms^2; \La_{\mcq^+,\infty}).
 \end{gather*} 
 
 Moreover,  the principal symbol of $\Xi_{1,1,1,0}$ at a point $(s,\omega, \mu, \varkappa) \in \La_{\mcq^-,\infty}$ away from $\Sigma_{j,\infty}^\pm,$ $j=1,2,3,$   
 on determines $f^{(3)}(u_0(q)),$ where $(q, \tau,\xi)\in N^*\Gamma^-\setminus 0$ and $(s,\omega,\mu,\varkappa)\in \La_{\mcq^-,\infty}\setminus 0$ are connected by a unique bicharacteristic for $H_{\varrho}.$  Moreover, by varying $z_j$ and $s_j,$ one determines  $f^{(3)}(u_0(q))$ for all $q\in \mr^4.$
\end{prop}
\begin{proof}   We are interested in the singularities of $w_{1,1,1,0}$ emanating from $\Gamma_-.$ As above, let $y_j= t-s_j+|z-z_j|,$ $j=1,2,3,$ then
$\Gamma_- =\{y_1=y_2=y_3=0\},$ which corresponds to the intersection of the three waves.   If  $\La_{\mcq^-}$ denote the Lagrangian submanifold of $T^*\mr^4\setminus 0,$ obtained by the flow-out  of the submanifold
 \begin{gather*}
 \La_0^-=(N^*\Gamma^-\setminus 0) \cap \{p=0\}, \;\ \text{ where }  p \text{ is the principal symbol of } \square,
 \end{gather*}
 under $H_p,$ the Hamilton vector field of $p.$  Therefore, the principal part of the singularity of $w_{1,1,1,0}$ on $\La_{\mcq^-}$ and away from the incoming surfaces and $\Gamma^-$ is given by
 \begin{gather*}
 E_+\left[ \left(\mcz f^{(3)}(u_0)\right)|_\Gamma \; y_{1+}^my_{2+}^my_{3+}^m\right].
 \end{gather*}
 
  It's symbol $\sigma(w_{1,1,1,0})$ satisfies,
 \begin{gather*}
 H_p \sigma(w_{1,1,1,0})=0, \\ 
 \sigma(w_{1,1,1,0})(q,\eta) = \frac{1}{3!}\mcz(q) f^{(3)}(u_0(q))\, \varsigma(q,\eta) \text{ for } (q,\eta)\in (N^*\Gamma^-\setminus 0)\cap \{p=0\}, \\ 
 \varsigma(q,\eta) \in S^{-3m-3} \text{ is elliptic } 
 \end{gather*}
 where, as above $p=\sigma_2(\square).$   This means the symbol is constant along the bicharacteristics of $H_p.$  This means we can read 
$\frac{1}{3!}f^{(3)}(u_0(q))$ from the singularities of $\mcn_+(w_{1,1,1,0}).$  Notice that one can also determine if $f^{(3)}(u_0(q))=0$ because there will be no singularities of top order.

One can do this computation quite explicitly for a particular, yet general enough example.  We pick $s_1=s_2=s_3=0,$ and 
 \begin{gather}
 z_1=(0,0,0,0), \;\ z_2=(2a,0,0,0) \text{ and } z_3=(0,2b,0,0), \label{def-z1z2z3}
 \end{gather}
 Then, the light cones with vertices on these points satisfy
 \begin{gather*}
 x_1^2+x_2^2+ x_3^2=t^2, \\
 (x_1-2a)^2+x_2^2+x_3^2=t^2,\\
 x_1^2+(x_2-2b)^2+x_3^2=t^2,
 \end{gather*}
 and they will intersect transversally at the hyperbolas
 \begin{gather*}
 \Gamma^-= \{ (a,b,x_3, t), \;\  t=-(x_3^2+ a^2+b^2)^\ha\}, \\
  \Gamma^+= \{ (a,b,x_3, t), \;\  t=(x_3^2+ a^2+b^2)^\ha\}.
 \end{gather*}
 For $t<0,$ the conormal bundle to $\Gamma^-$ is given by
 \begin{gather*}
 N^*\Gamma^-=\{ x_1=a, x_2=b, t=-(x_3^2+a^2+b^2)^\ha, x_3\tau + t \xi_3=0\}.
 \end{gather*}
 The Lagrangian submanifold $\La_{\mcq^-}$ obtained by the flow-out of $(N^*\Gamma^-\setminus 0) \cap \{p=0\},$ is given by
 \begin{gather*}
 \xi_1=\xi_{10}, \; \xi_2=\xi_{20}, \;  \xi_3=\xi_{30}, \;\ \tau=\tau_0, \;\ x_{30}\tau_0+ t_0 \xi_{30}=0, \\ t_0= -(x_{30}^2+a^2+b^2)^\ha,  \text{ and } \tau_0=|\xi_0|, \\
 x_1=a -2\xi_{10} \nu, \;\ x_2=b -2\xi_{20}\nu, \;\ x_3= x_{30}-2\xi_{30}\nu, \;\ t=t_0+2\tau_0 \nu, \;\ \nu\in \mr.
 \end{gather*}
 
 The projection of $\La_{Q^-}$ to $\mr^4$ is denoted by $\mcq^-$ and it is given by 
 \begin{gather}
\mcq^-=\{z=(t, z):  |z-(a,b,x_{30})|^2=(t-t_0)^2, \; x_{30}t = x_3 t_0, \; t_0= -(x_{30}^2+a^2+b^2)^\ha\}.\label{eq-mcq}
\end{gather}

If one writes $z=r\omega,$ $r=|z|,$ and sets $t=s+r,$ then
\begin{gather}
\begin{gathered}
(r,s,\omega)\in \mcq^- \text{ if and only if } s^2-2s t_0 +2r\left( s+\lan \omega, (a,b,x_{30})\ran-t_0\right)=0, \\ x_{30}(r+s)=t_0 r \omega_3, \text{ and } t_0=-( x_{30}^2+a^2+b^2)^\ha.
\end{gathered}\label{eq-mcq-rrad}
\end{gather}
If we divide the equation of $\mcq^-$ by $r$ and for fixed $s,$ let $r\rightarrow \infty,$ one obtains:
\begin{gather}
\mcq_\infty^-= \{ (s,\omega): s+\lan \omega, (a,b,x_{30})\ran=t_0, \;\  \omega_3= \frac{x_{30}}{t_0}, \; t_0=-( x_{30}^2+a^2+b^2)^\ha\}.
\label{eq-mcq-infty}
\end{gather}

 Given the choices of $z_j,$ $j=1,2,3,$ we have
 \begin{gather*}
 \Sigma_{1,\infty}^+= \{s=0\}, \\
 \Sigma_{2,\infty}^+=\{ s- \lan\omega, (a,0,0)\ran =0\}, \\
  \Sigma_{3,\infty}^+=\{ s- \lan\omega, (0,b,0)\ran=0\}.
  \end{gather*}
 
    This shows that, there are many points on $\mcq^-_{\infty}$ which are not on $\Sigma_{j,\infty}^+.$ So by computing the principal symbol of $w_{1,1,1,0}$  along the null bicharacteristics for $p$ starting over a point
 $(a, b, x_{30},t_0)\in \Gamma^-,$ which do not  lie on $N^*\Sigma_j\setminus 0,$ $j=1,2,3,$
$\sigma(w_{1,1,1,0})$ determines $\left(f^{(3)}(u_0)\right)(a,b, x_{03},t_0),$ with  $t_0=-(x_{03}^2+a^2+b^2)^\ha.$ By varying $s_1=s_2=s_3=s^*,$ one then determines
$\left(f^{(3)}(u_0)\right)(a,b, x_{03},t_0),$ with  $t_0-s^*=-(x_{03}^2+a^2+b^2)^\ha.$   By varying $a,$  $b$ and $s^*$ one determines  $f^{(3)}(u_0)(z,t)$ for all $(z,t)\in \mr^4.$
 \end{proof}

\section{Singularities Produced by the Interaction of Four Waves}

In this section we consider the case where $\Ups_j,$ $j=1,2,3,4,$ are given by \eqref{def-upsj} and we are interested in analyzing the singularities of $w_{1,1,1,1}.$  In this case we have to consider a system of equations. The terms $w_{\beta},$ with $|\beta|=1:$
\begin{gather*}
\beta=(1,0,0,0), \; (0,1,0,0), \;\ (0,0,1,0), \; (0,0,0,1),
\end{gather*}  
satisfy the linear equation \eqref{defwj}, the terms $w_\beta,$ of order two, 
\begin{gather*}
\beta=(1,1,0,0), \; (1,0,1,0), \; (1,0,0,1),\; (0,1,1,0), \; (0,1,0,1), \; (0,0,1,1),
\end{gather*}
 satisfy \eqref{defwjk}. The terms of order three, $w_{\beta},$ 
 \begin{gather*}
 \beta=(1,1,1,0), \; (1,1,0,1), \; (1,0,1,1), \; (0,1,1,1),
 \end{gather*}
  satisfy \eqref{defw123}. To express the equation \eqref{defw1234} for the term $w_{1,1,1,1},$  we split the terms $\alpha=(1,1,1,1)$ into
\begin{gather*}
\alpha=\beta_1+\beta_2, \;\ |\beta_1|=2, \ |\beta_2|=2, \\
\alpha=\beta_1+\beta_2, \;\ |\beta_1|=1, \; |\beta_2|=3, \\
\alpha=\beta_1+\beta_2+\beta_3, \;  |\beta_j|=1, \; j=1,2, \; |\beta_3|=2, \\
\text{ and } \alpha= (1,0,0,0)+(0,1,0,0)+(0,0,1,0)+(0,0,0,1),
\end{gather*}
and write
\begin{gather}
\begin{gathered}
\square w_{1,1,1,1}= -f'(u_0) w_{\alpha}- \frac{ 1 }{2!}  f^{(2)}(u_0)\sum_{|\beta_1|=1, |\beta_2|=3} w_{\beta_1} w_{\beta_2}
-\frac{1}{2!}  f^{(2)}(u_0)\sum_{|\beta_1|=2, |\beta_2|=2} w_{\beta_1} w_{\beta_2}+ \\
\frac{1}{3!}  f^{(3)}(u_0)\sum_{|\beta_1|=1, |\beta_2|=1, |\beta_3|=2} w_{\beta_1} w_{\beta_2}w_{\beta_3}+ 
\frac{1}{4!}  f^{(4)}(u_0) w_{1,0,0,0}w_{0,1,0,0}w_{0,0,1,0}w_{0,0,0,1}, \\
 \mcn_- w_{1,1,1,1}=0.
\end{gathered}\label{sing4w}
\end{gather}

We will compute the principal symbol of $w_{1,1,1,1}$ on the light cone emanating from $\ga_-=\cap_{j=1}^4 \Sigma_j^-.$   
Here we have the interaction of four waves and we need to describe the singularities coming from each term of \eqref{sing4w}. 
Just like the triple interaction case discussed above, this analysis was done by Lassas, Uhlmann and Wang in Proposition 3.11 of \cite{LUW}.
\begin{prop}\label{reg-w1111-N} ( Lassas, Uhlmann and Wang \cite{LUW}) Let $\gamma_\pm=\Sigma_1^\pm\cap \Sigma_2^\pm \cap \Sigma_3^\pm\cap \Sigma_4^\pm,$  and let $\La_\pm \subset T^*\mr^4\setminus 0$ denote the Lagrangian submanifold obtained by the flow-out of $(T_{\ga_\pm}^*\mr^4\setminus 0) \cap\{p=0\}$ by $H_p.$  If $f^{(4)}(u_0(\ga_\pm))\not=0,$  then away from $N^* \Sigma_j^\pm$ and $T^*{\ga_\pm}\mr^4,$ $w_{1,1,1,1}\in I^{-4m-\frac92}(\mr^4, \mcq^\pm)$ and principal symbol of $w_{1,1,1,1}$ is given by $\frac{1}{4!}E_+\left( \mcw (f^{(4)}(u_0(\ga_\pm)) y_{1+}^my_{2+}^my_{3+}^m y_{4+}^m\right),$ where $\mcw=Z_1Z_2Z_3Z_4.$  
\end{prop}

 The key point here is that there finitely many interactions, and $\Xi_{1,1,1,1}$ is conormal to a finite set of surfaces and we are interested in the singularity of $\Xi_{1,1,1,1}=\mcn_+(w_{1,1,1,1})$ on the radiation pattern of the light cone over the quadruple interaction $\cap_{j=1}^4 \Sigma_j^-,$ which are not on the other surfaces.  
 
The following is the main result of this section:
\begin{prop}\label{reg-w1111} Let $\ga_-=\cap_{j=1}^4 \Sigma_j^-$ and let $\La_-\subset T^*\mr^4\setminus 0$ denote the Lagrangian submanifold obtained by the flow-out of $(T_{\ga_-}^*\mr^4\setminus 0) \cap\{p=0\}$ by $H_p$ and let $\La_{-,\infty}$ denote its radiation pattern.   Let $\La_{\mcq^{jkl}_{\pm},\infty}$ be the forward radiation pattern of $\La_{\mcq_\pm^{jkl}},$  which are Lagrangians submanifolds emanating from the triple interactions $\Sigma_j\cap \Sigma_k\cap \Sigma_l,$ and let $N^*\Sigma_{j\infty},$ $j=1,2,3,4,$ be as defined in Proposition \ref{new-sing-111}. Then, away from $\La_{\mcq^{jkl}_{\pm},\infty}$ and from   $N^*\Sigma_{j\infty},$ $j=1,2,3,4,$
$\mcn_+(w_{1,1,1,1})=I^{-4m-\frac{17}4}(\mr\times \ms^2, \La_{-\infty}).$  In this case, if $\ga_-=(t_0,z^*),$ 
the principal symbol of $w_{1,1,1,1}$ at a point of $\La_{-\infty}$ which is not on any other Lagrangians, determines $f^{(4)}(u_0(\ga_-)).$ 
\end{prop}

We can apply this to the four spherical waves given by $s_1=s_2=s_3=s_0=0,$ $z_1,z_2,z_3$ given by \eqref{def-z1z2z3} and $z_4=(0,0,2c,0).$ The four waves will intersect at
\begin{gather*}
\gamma_-=\{ x_1=a, x_2=b, x_3=c, t=t_{0+}=-(a^2+b^2+c^2)^\ha\} \;\ \text{ and } \\ \gamma_+=\{ x_1=a, x_2=b, x_3=c, t=t_{0-}=-(a^2+b^2+c^2)^\ha\}.
\end{gather*}
The forward cones $\mcq_\pm$ are given by
\begin{gather*}
\mcq_\pm= \{t-t_{0\pm}= |z-\gamma_\pm|\}.
\end{gather*}
Their radiation patterns are given by
\begin{gather*}
\mcq_{\pm,\infty}= \{s-t_{0\pm}- \lan \omega, \gamma_\pm\ran=0\}.
\end{gather*}
By following the argument used in the case of the interaction of three waves, this determines $f^{(4)}(u_0(a,b,c,t_{0-})).$ Again by varying 
$s^*=s_1=s_2=s_3=s_4,$ this determines $$f^{(4)}(u_0(a,b,c,s^*+t_{0-})),$$ and hence $f^{(4)}(u_0(p))$ for all $p\in \mr^4.$

\section{Acknowledgements}  
      S\'a Barreto  and Uhlmann were visiting members of the Microlocal Analysis Program of the Mathematical Sciences Research Institute (MSRI) in Berkeley, California,  in the fall 2019, when part of this work was done. Their membership at the MSRI was supported by the National Science Foundation Grant No. DMS-1440140.
  
  A, S\'a Barreto is grateful to the Simons Foundation  for their support under grant \#349507, Ant\^onio S\'a Barreto. 

G. Uhlmann was partially supported by NSF, a Walker Family Endowed Professorship at UW and a Si-Yuan Professorship at IAS, HKUST.

%%===============================REFERENCES==========================================%

\end{document}